\documentclass[reqno,11pt,20pt]{amsart}
\usepackage[ansinew]{inputenc}
\usepackage[english]{babel}
\usepackage[all]{xy}
\usepackage{amsmath,latexsym,amssymb,verbatim}
\input arrow.tex

\newcommand{\OO}{{\mathcal O}}

\newcommand{\codim}{{\operatorname{codim}}}

\newcommand{\LL}{{\mathbb L}}
\newcommand{\FF}{{\mathtt F}}
\newcommand{\M}{{\mathcal M}}
\newcommand{\I}{{\mathcal I}}

\newcommand{\Lc}{{\mathcal L}}
\newcommand{\NN}{{\mathbb N}}
\newcommand{\ZZ}{{\mathbb Z}}
\newcommand{\HH}{{\mathbb H}}

\newcommand{\PP}{\mathbb P}
\newcommand{\RR}{\mathbb R}
\newcommand{\0}{\mathbf 0}
\newcommand{\un}{\mathbf 1}

\newcommand{\mm}{{\mathfrak m}}

\DeclareMathOperator{\limi}{{lim}}

\newcommand{\Spec}{\operatorname{Spec}}
\newcommand{\Proj}{\operatorname{Proj}}

\newcommand{\enumera}{\begin{enumerate}}
\newcommand{\eenumera}{\end{enumerate}}
\newcommand{\C}{{\mathcal C}}

\newcommand{\A}{{\mathbb A}}

\newcommand{\stella}{{\scriptscriptstyle\bigstar}}
\DeclareMathOperator{\Ker}{{Ker}} \DeclareMathOperator{\Ima}{{Im}}
\DeclareMathOperator{\Hom}{{Hom}}
\DeclareMathOperator{\Kos}{{Kos}}
\DeclareMathOperator{\Res}{{Res}}
\DeclareMathOperator{\di}{{d}}

\newcommand{\ilim}[1]{\,\underset{#1}{\underset{\to}{\limi}}\,}
\newcommand{\alineas}[1]{\begin{array}{#1}}
\newcommand{\alinea}{\begin{array}{l}}
\newcommand{\ealinea}{\end{array}}
\newcommand{\ealineas}{\end{array}}

\newcommand{\pun}{{\scriptscriptstyle \bullet}}

\newcommand{\HHom}{\mathcal{H}om}
\DeclareMathOperator{\Link}{Link}
\newcommand{\HExt}{\mathcal{E}xt}

\newcommand{\Qcoh}{\operatorname{Qcoh}}
\newcommand{\Perf}{\operatorname{Perf}}
\newcommand{\Coh}{\operatorname{Coh}}

\newcommand{\Shv}{\operatorname{Shv}}

\newcommand{\supp}{\operatorname{supp}}

\DeclareMathOperator{\Ext}{{Ext}}

\DeclareMathOperator{\Coker}{{Coker}}

\DeclareMathOperator{\id}{{id}}

\theoremstyle{plain}
\newtheorem{thm}{Theorem}[section]
\newtheorem{lem}[thm]{Lemma}
\newtheorem{cor}[thm]{Corollary}
\newtheorem{prop}[thm]{Proposition}
\newtheorem{defn}[thm]{Definition}
\newtheorem{rem}[thm]{Remark}
\newtheorem{rems}[thm]{Remarks}

\newtheorem*{ex}{Example}
\newtheorem*{exs}{Examples}

\newtheorem{cosa}[thm]{}

\numberwithin{equation}{thm}

\begin{document}

\title{A geometrization of Stanley--Reisner theory}

\author{Fernando Sancho de Salas and Alejandro Torres Sancho}

\address{ Departamento de
Matem\'aticas and Instituto Universitario de F\'isica Fundamental y Matem\'aticas (IUFFyM)\newline
Universidad de Salamanca\newline  Plaza de la Merced 1-4\\
37008 Salamanca\newline  Spain}
\email{fsancho@usal.es}
\email{atorressancho@usal.es}



\thanks {Work supported by Grant PID2021-128665NB-I00 funded by MCIN/AEI/ 10.13039/501100011033 and, as appropriate, by ``ERDF A way of making Europe''.}
 
\begin{abstract}   We give a geometric interpretation of the Stanley--Reisner correspondence, extend it to schemes, and interpret it in terms of the field of one element.
\end{abstract}



\maketitle

\section*{Introduction}

Stanley--Reisner theory provides a bridge between combinatorics and commutative algebra. Reisner's criterion for Cohen-Macaulayness and Stanley's proof of the Upper Bound Conjecture for simplicial spheres were undoubtedly two landmarks  in its beginnings. Since then, the correspondence between simplicial complexes and squarefree monomial ideals has been responsible for substantial progress in both fields (see \cite{FMS} for a survey). Very briefly, in this paper we provide a geometric interpretation of the Stanley--Reisner correspondence that leads to an interplay between ordinary schemes and $\FF_1$-schemes. Let us be more precise. Let us start by reformulating the Stanley--Reisner correspondence.\smallskip

 Let us fix a ring $k$. Let us denote $\Delta_n :=\{1,\dots,n\}$ the set of $n$ elements and let 
$$\mathbb{A}^n_{\mathtt{F}_1}\,=\, \{\text{subsets of }\Delta_n\}.$$
It is a finite topological $T_0$-space (in other words, a finite poset), with the topology associated to the partial order given by inclusion. The empty subset of $\Delta_n$ is an element of $\A^n_{\FF_1}$, denoted by $\0$. A finite abstract simplicial complex $K$ (with at most $n$ vertices) is nothing but a closed subset $K\subseteq\mathbb{A}^n_{\mathtt{F}_1}$. According to Stanley--Reisner theory (see \cite{St}, \cite{BH}), such a closed subset defines a monomial ideal $I_K$ of $k[x_1\dots,x_n]$, hence a closed subscheme $S_K$ of the affine space $\mathbb{A}^n_k=\text{Spec}\,k[x_1,\dots,x_n]$,
$$S_K:=\,\text{Spec}\,k[x_1,\dots,x_n]/I_K\,\subseteq\,\mathbb{A}^n_k.$$

We reformulate the Stanley--Reisner correspondence as follows. Let us denote
$$\begin{aligned}\mathrm{Shv}_k(\mathbb{A}^n_{\mathtt{F}_1}) &:=\, \text{category of sheaves of }k\text{-modules on }\mathbb{A}^n_{\mathtt{F}_1}\\
\text{Qcoh}(\mathbb{A}^n_k ) &:=\, \text{category of quasi-coherent modules on }\mathbb{A}^n_k.
\end{aligned}$$
We shall construct, in a natural way, a continuous map 
$$ 
\pi\colon\mathbb{A}^n_k \longrightarrow \mathbb{A}^n_{\mathtt{F}_1}
$$
and an exact $k$-linear functor
$$ \pi^{ \scriptscriptstyle\bigstar} : \text{Shv}_k(\mathbb{A}^n_{\mathtt{F}_1}) \longrightarrow \text{Qcoh}(\mathbb{A}^n_k) $$ 
such that, for any closed subset $K\subseteq \mathbb{A}^n_{\mathtt{F}_1}$ one has:
$$\pi^{-1}(K)\,=\, S_K\quad(\text{topologically})\qquad ,\qquad \pi^{\stella}(k_K)\,=\,\mathcal{O}_{S_K}:=\,\mathcal{O}_{\mathbb{A}^n_k}/I_K $$
where $k_K$ is the constant sheaf $k$ supported on $K$. 
Finally, $\pi^\stella$ has an exact $k$-linear right adjoint
\[ \pi_\stella\colon \Qcoh(\A^n_{k})\to \Shv_k(\A^n_{\FF_1}).\]


The functor $\pi^\stella$ is not new. If we consider the equivalence $\Qcoh(\A^n_k)=\{ R\text{-modules}\}$, with  $R=k[x_1,\dots,x_n]$, then $\pi^\stella (F)$ is the squarefree  $R$-module associated to $F$ by Yanagawa's equivalence (\cite[Thm. 3.3]{Ya01}), at least if $F$ is a sheaf of finite dimensional vector spaces over a field $k$. This explains the relevance of squarefree modules in Stanley--Reisner theory. Our construction of $\pi^\stella$   differs from Yanagawa's, and it is based in a general Yoneda construction that will be needed for a more general framework (see below) and  may be of interest for other posets. 

Our point of view is to interpret $\pi\colon \A^n_k\to\A^n_{\FF_1}$ as a ``morphism of schemes'', where $\pi^\stella$ and $\pi_\stella$ 
play  the role of inverse and direct  images, and $\A^n_{\FF_1}$ is thought as the affine $n$-dimensional $\FF_1$-scheme. It is a fact that, topologically, $\A^n_{\FF_1}$ agrees with Deitmar's affine scheme (\cite{Deitmar}), i.e.,
\[ \A^n_{\FF_1}=\Spec \FF_1[x_1,\dots, x_n]\] where $\FF_1[x_1,\dots,x_n]$ is the free abelian monoid generated by $x_1,\dots,x_n$ (also containing an extra element $0$) and the topology of $\Spec \FF_1[x_1,\dots,x_n]$ is the Zariski topology. More importantly, the category of sheaves on $\A^n_{\FF_1}$ has a remarkable similarity (as we shall try to show) to the category of quasi-coherent modules on an ordinary affine scheme $\A^n_k$. As a first example of this similarity we shall see that Grothendieck's theory of dualizing complexes and Cohen-Macaulayness can be mimicked for the category of sheaves on $\A^n_{\FF_1}$. Thus, we shall see that $\mathbb{A}^n_{\mathtt{F}_1}$ has a canonical sheaf $\omega_{\mathbb{A}^n_{\mathtt{F}_1}}$ which is dualizing (i.e., yields reflexivity, Proposition  \ref{duality-1}) and that dualizes local cohomology (Proposition  \ref{duality-2}). Following Grothendieck, we shall say that a finitely generated sheaf $F$ on $\A^n_{\FF_1}$ is {\em Cohen--Macaulay} if it is $k$-flat and $\RR\HHom_{\A^n_{\FF_1}}^\pun(F,\omega_{\A^n_{\FF_1}})$ is a $k$-flat sheaf (placed in some degree). For any closed subset $K$ of $\A^n_{\FF_1}$,  we shall obtain a canonical complex $\omega_K^\pun$ and we shall prove that  classic Cohen-Macaulayness (defined in terms of the homology of links) coincides with Grothendieck's one, and that the homology of links has an intrinsic meaning in terms of the canonical complex of $K$. More precisely we prove the following theorem (that summarizes several propositions in the paper):

\begin{thm}\label{1} Let $K\subseteq \A^n_{\FF_1}$ be a simplicial complex of codimension $d$, and let us denote
\[ \omega_K^\pun=\RR\HHom_{\A^n_{\FF_1}}^\pun (k_K,\omega_{\A^n_{\FF_1}})_{\vert K}.\] Then:

\begin{enumerate} \item  $ \omega_K^\pun$ is a dualizing complex, i.e., for any $F\in D_\text{\rm perf}(K)$, one has an isomorphism 
\[ F\overset\sim \longrightarrow  \RR\HHom_{K}^\pun(\RR\HHom_{\A^n_{\FF_1}}^\pun(F,\omega_K^\pun),\omega_K^\pun)\]

\item   $ \omega_K^\pun$ dualizes local cohomology: for any $F\in D(K)$ one has:
\[ \RR\Hom_K^\pun(F, \omega_K^\pun)=\RR\Hom_k^\pun (\RR\Gamma_\0(K,F), k)[-n]\]

\item For any $p\in K$ one has:
\[    H^i(\omega_K^\pun)_p= \widetilde H_{n-c_p-i-1}(\Link_K(p),k),\qquad c_p=  \text{  dimension of the closure of }p.\]
\item The following conditions are equivalent:
\begin{enumerate}\item $K$ is Cohen--Macaulay, i.e.: $\widetilde H_i(\Link_K(p),k)=0$ for any $p\in K$ and any $i<\dim\Link_K(p)$.
\item $\omega_K^\pun$ is a sheaf, i.e., $\omega_K^\pun=\omega_K[-d]$ for some sheaf $\omega_K$ (named canonical sheaf of $K$), whose stalkwise description is
\[ (\omega_K)_p= \widetilde H_{m_p}(\Link_K(p),k),\quad m_p=\dim\Link_K(p).\]
\item $\HExt_{\A^n_{\FF_1}}^i(k_K,\omega_{\A^n_{\FF_1}})=0$ for any $i\neq d$.
\end{enumerate}
\end{enumerate}
\end{thm}

Some parts of this theorem are related to \cite[Thm. 4.9]{Ya03} (see also \cite{Mil}, \cite{Ya02}), where $\A^n_{\FF_1}$ and $K$ are replaced by their geometric realizations. Notice that, for Stanley--Reisner theory, it is much better to work on $\A^n_{\FF_1}$ than on its geometric realization $\vert\A^n_{\FF_1}\vert$, since we have  a continuous map $\pi\colon\A^n_k\to\A^n_{\FF_1}$ which does not lift to $\vert\A^n_{\FF_1}\vert$, because any continuous map $\A^n_k\to \vert\A^n_{\FF_1}\vert$ is constant.\medskip

Our reformulation of the Stanley--Reisner correspondence can be done in a more general framework, where the affine scheme $\A^n_k$ is replaced by an arbitrary $k$-scheme, in the following way. Given a $k$-scheme $S$ and a sequence $D_1,\dots,D_n$ of $n$ effective Cartier divisors in $S$, we define a continuous map
$$ f :S \longrightarrow \mathbb{A}^n_{\mathtt{F}_1}$$
and a right exact $k$-linear functor
$$  f^{\stella} : \text{Shv}_k(\mathbb{A}^n_{\mathtt{F}_1}) \longrightarrow \text{Qcoh}(S)$$
such that
$$f^{-1}(H_i)\,=\, D_i\quad(\text{topologically})\qquad,\qquad f^{\stella}(k_{H_i})=\mathcal{O}_{D_i}$$
where the $H_i=\{\text{subsets of }(\Delta_n-\{i\})\}$ are the standard hyperplanes of $\mathbb{A}^n_{\mathtt{F}_1}$. Moreover, for any closed subset $K\subseteq\mathbb{A}^n_{\mathtt{F}_1}$ we have $f^{\stella}(k_K)=\mathcal{O}_{S_K}$ for some closed subscheme $S_K\subseteq S$, whose underlying topological space is $f^{-1}(K)$. 
Finally, $f_\stella$ is defined as the right adjoint of $f^\stella$. 

We shall say that the pair $(f\colon S\to \mathbb{A}^n_{\mathtt{F}_1},\,f^{\stella})$ is the {\it algebraic $k$-morphism} associated to the sequence $D_1,\dots,D_n$. The particular case of the Stanley--Reisner correspondence $\pi^\stella$ mentioned above is  important because any algebraic morphism $f\colon S\to \A^n_{\FF_1}$ factors, locally, as the composition of a morphism of schemes $g\colon S\to\A^n_k$ and $\pi\colon\A^n_k\to\A^n_{\FF_1}$. Moreover $g$ is flat if  $f^\stella $ is exact. 

Given an algebraic morphism $f\colon S\to\A^n_{\FF_1}$, the functors $f^\stella $ and $f_\stella $ are derived to obtain adjoint functors $\LL f^\stella$ and $\RR f_\stella$. The first result (that agrees with viewing $f$ as a ``morphism of schemes'') is that $\RR f_\stella$ respects cohomology:
\[\RR\Gamma(S, \M)=\RR\Gamma(\A^n_{\FF_1},\RR f_\stella\M)\] for any $\M\in D(S)$. If $S$ is a proper $k$-scheme, then $\RR f\stella$ satisfies the finiteness theorem, that is, $\RR f_\stella$ maps $D^b_c(S)$ into $D^b_c(\A^n_{\FF_1})$.

For any $F,G\in D(\A^n_{\FF_1})$ one has a natural morphism
\[\LL f^\stella \RR\HHom_{\A^n_{\FF_1}}^\pun(F,G)\longrightarrow \RR\HHom_S^\pun(\LL f^\stella F,\LL f^\stella G)\] which is not (in fact, almost never!) an isomorphism. Hence, the following fundamental result is a kind of miracle that explains the crucial role of the dualizing sheaf $\omega_{\A^n_{\FF_1}}$:

\begin{thm}[Thm. \ref{extens}]\label{2} For any $F\in D_\text{\rm perf}(\A^n_{\FF_1})$, the natural morphism
\[\LL f^\stella \RR\HHom_{\A^n_{\FF_1}}^\pun(F,\omega_{\A^n_{\FF_1}})\longrightarrow \RR\HHom_S^\pun(\LL f^\stella F,  f^\stella \omega_{\A^n_{\FF_1}})\] is an isomorphism.
\end{thm}

As we shall see, this result, in the case of the Stanley--Reisner correspondence $\pi^\stella$, is at the root of some fundamental results of Stanley--Reisner theory like Reisner's criterion and the structure of the canonical complex of a Stanley--Reisner ring. For example, \cite[Thm. 2.6]{Ya00} is an immediate consequence and holds on an arbitrary ring $k$. A  result related to Theorem \ref{2}, stated in terms of the geometric realization of $\A^n_{\FF_1}$, may be found in \cite[Thm. 4.2]{Ya03}.

Combining Theorems \ref{1} and \ref{2}, we obtain the generalization of  Reisner's criterion for any algebraic morphism and any sheaf on $\A^n_{\FF_1}$. We say that $f\colon S\to\mathbb{A}^n_{\mathtt{F}_1}$  is flat  if $f^{\stella}$ is an exact functor. If moreover $\,f^{\stella}(\mathcal M)=0\, \Rightarrow\, \mathcal M =0$ (for any finitely generated $k$-module $\mathcal M$ on $\mathbb{A}^n_{\mathtt{F}_1}$), then we say that $f\colon S\to\mathbb{A}^n_{\mathtt{F}_1}$ is faithfully flat. Then, Reisner's criterion takes the following natural form (if $f$ is thought as a morphism of schemes):\medskip

\begin{thm}[Thm. \ref{ReisnerOnSchemes}]\label{3} Let $S$ be a $k$-scheme and $f\colon S\to\mathbb{A}^n_{\mathtt{F}_1}$ a faithfully flat algebraic $k$-morphism. Let $K\subseteq\mathbb{A}^n_{\mathtt{F}_1}$ be a closed subset and $S_K=f^{-1}(K)$ the associated closed subscheme of $S$. Let us denote $d =\mathrm{codim}\,(K)$. The following are equivalent:\smallskip

$(1)$   $K$ is Cohen-Macaulay.\smallskip


$(2)$ $\mathcal{E}xt^i_{\mathcal{O}_S}(\mathcal{O}_{S_K},\mathcal{O}_S)=0$ for any $i\neq d$.\smallskip

Moreover, if $S$ is (relatively) Gorenstein, then these conditions are equivalent to:\smallskip

$(3)$ The $k$-scheme $S_K$ is (relatively) Cohen--Macaulay.\smallskip

More generally, for any finitely generated $k$-module $F\in\mathrm{Shv}_k(\mathbb{A}^n_{\mathtt{F}_1})$, 
the following conditions are equivalent:\smallskip

$(1)$ $F$ is Cohen--Macaulay (Definition \ref{CM-sheaf}).\smallskip

$(2)$ There exists a non negative integer $d_0$, such that $\mathcal{E}xt^i_{\mathcal{O}_S}(f^{\stella}F,\mathcal{O}_S)=0$ for any $i\neq d_0$ and $\mathcal{E}xt^{d_0}_{\mathcal{O}_S}(f^{\stella}F ,\mathcal{O}_S)$ is flat over $k$.\smallskip

$(3)$ $f^{\stella}\mathcal F$ is (relatively) Cohen--Macaulay (assuming that $S$ is relatively Gorenstein). \end{thm}

It is worthy to mention that, for classic Stanley--Reisner correspondence, we also obtain the following result. If $K'\subseteq K$ is a subcomplex, then $K-K'$ is Cohen--Macaulay (in our Grothendieck's sense) if and only if the scheme $S_K- S_{K'}$ is (relatively) Cohen--Macaulay. Notice that $K-K'$ is no longer a simplicial complex, but there is a general notion of Cohen--Macaulayness on finite posets due to Baclawski (\cite{Ba}). This  notion, applied to $K-K'$, is stronger than ours and does not satisfy the above result. Hence, our sheaf-theoretic Cohen--Macaulayness notion fits better with Stanley--Reisner correspondence than Baclawski's. See Remark \ref{remark-opensimplicial}  for details and \cite{ST} for a comparison of these two notions of Cohen--Macaulayness on finite posets.\medskip

Regarding the canonical complex of a Stanley--Reisner ring we shall obtain the following result, as a consequence of Theorem \ref{2}. Let $K\subseteq \A^n_{\FF_1}$ and $S_K$ its associated Stanley--Reisner scheme. The functor $\pi^\stella$ induces a functor
\[ \pi_K^\stella\colon \Shv_k(K)\to \Qcoh(S_K)\] and we shall prove:

\begin{thm}[Thm. \ref{canonicalcomplex}]\label{7} $\pi_K^\stella\, (\omega_K^\pun)= \omega_{S_K}^\pun:= $restriction to $S_K$ of $\RR\HHom_{\A^n_k}^\pun(\OO_{S_K},\Omega^n_{\A^n_k/k})$, which is a (relative) dualizing complex of $S_K$.
\end{thm}

Notice that $\omega_K^\pun$ is explicitely computed by $$\omega_K^\pun=\RR\HHom_{\A^n_{\FF_1}}^\pun( k_K,\omega_{\A^n_{\FF_1}})_{\vert K}=\HHom_{\A^n_{\FF_1}}(\Kos_\pun(k_K),\omega_{\A^n_{\FF_1}})_{\vert K}$$  hence Theorem \ref{7} gives an explicit combinatorial description of the canonical complex of $S_K$, which improves that of \cite{Gr} and other results in the literature. A related result in terms of the geometric realization of $K$ is given in  \cite[Cor. 4.3]{Ya03}. Of course, if $K$ is Cohen--Macaulay, then $\pi_K^\stella\, (\omega_K)$ is the (relative) canonical sheaf of the Cohen--Macaulay scheme $S_K$.\medskip

We have also obtained some results for the projective version of Stanley--Reisner correspondence. Let us denote $\PP^n_{\FF_1}:=\A^{n+1}_{\FF_1}-\0$, which agrees with the underlying topological space of Dietmar's projective $\FF_1$-scheme (and, more importantly, the category of sheaves on $\PP^n_{\FF_1}$ behaves very similarly to the category of quasi-coherent sheaves on an ordinary projective scheme). Let $\PP^n_k=\Proj k[x_0,\dots,x_n]$ be the ordinary projective scheme over $k$. We shall construct a continuous map $\pi\colon \PP^n_k\to \PP^n_{\FF_1}$, an exact functor
\[\pi^\stella\colon \Shv_k(\PP^n_{\FF_1})\longrightarrow \Qcoh(\PP^n_k)\] and its right adjoint $\pi_\stella$. For each closed subset $K$ of $\A^{n+1}_{\FF_1}$, let us denote $K^*=K-\0$, which is a closed subset of $\PP^n_{\FF_1}$. Then one has:
\[ \pi^{-1}(K^*)=\PP(S_K) \,\,\, (\text{\rm topologically}),\quad \pi^\stella (k_{K^*})=\OO_{\PP(S_K)},\]  where $\PP(S_K)$ denotes the projectivization of the Stanley--Reisner scheme $S_K$ associated to $K$. All the results obtained for the classic Stanley--Reisner correspondence may be repeated for the projective version. We shall just mention some relevant ones. The first one is that 
\[\RR\pi_\stella (\omega_{\PP^n_k})=\omega_{\PP^n_{\FF_1}} \] where $\omega_{\PP^n_k}=\Omega^n_{\PP^n_k/k}$ is the (relative) canonical sheaf of $\PP^n_k$ and $\omega_{\PP^n_{\FF_1}}$ is the canonical sheaf of $\PP^n_{\FF_1}$ (it is the restriction to $\PP^n_{\FF_1}$ of the canonical sheaf of $\A^{n+1}_{\FF_1}$, but, also, it is the sheaf that dualizes cohomology). This result, together with basic duality theory, yields that $\pi^\stella$ preserves cohomology:
\[ H^i(\PP^n_k,\pi^\stella F)=H^i(\PP^n_{\FF_1}, F) \] for any $F\in D(\PP^n_{\FF_1})$. Finally, for each simplicial complex $K$, the complex $\omega_{K^*}^\pun:=(\omega_K^\pun)_{\vert K^*}$ is a canonical complex on $K^*$ that dualizes cohomology and $\pi_{K^*}^\stella\,(\omega_{K^*}^\pun)$ is the dualizing complex of the projective scheme $\PP(S_K)$, where $\pi_{K^*}^\stella\colon\Shv(K^*)\to\Qcoh(\PP(S_K)$ is the functor induced by $\pi^\stella$.
\medskip

\medskip

The paper is divided as follows.

Section \ref{prelims} is devoted to recall elementary properties about finite (and $T_0$) topological spaces (i.e., finite posets) and sheaves (of $k$-modules) on them and to fix the notations we shall use. We shall recall the standard resolutions of a sheaf and introduce the Koszul complex associated to a closed subset. We shall also recall elementary properties of the finite spaces $\A^n_{\FF_1}$ and $\PP^n_{\FF_1}$.

Section \ref{Yoneda} is devoted to the construction of the functor $f^\stella$ in a general form (Theorem \ref{extension}), i.e., for any finite topological space $X$ and any $k$-linear abelian category $\C$. It is just a slight modification of the Yoneda Extension Theorem (\cite{KS}).

In Section \ref{St-R-functor} we particularize the construction of section  \ref{Yoneda} to define the functor  $f^\stella$ associated to an algebraic morphism $f_D\colon S\to\A^n_{\FF_1}$ defined by a collection $D$ of $n$ effective Cartier divisors on $S$. We develop the most basic properties of $f^\stella$, its right adjoint $f_\stella$ and their derived versions. We make a study of  the flatness of $f_D$. The main result is Theorem \ref{thetheorem} which characterizes the flatness or fully flatness of $f_D$ in terms of geometric conditions on the divisors.

Section \ref{CM-section} is devoted to the sheaf-theoretic study - that mimics Grothendieck's ideas - of Cohen-Macaulayness on $\A^n_{\FF_1}$ as explained above.

In section \ref{Main results} we shall give the main results about Reisner's criterion and the structure of the canonical complex of a Stanley--Reisner ring that we have mentioned above.


Finally, throughout this paper, a scheme always means a noetherian scheme; hence, whenever we say a $k$-scheme, $k$ is also assumed to be noetherian.

\section{Preliminaries and notations}\label{prelims}

\subsection{Finite topological spaces}

\begin{cosa}\label{cosa1} {\rm Let $X$ be a finite $T_0$-topological space, i.e., a finite poset. For each $p\in X$ we shall denote
\[ \aligned U_p&=\text{ smallest open subset containing }p, \\ C_p&=\text{ closure of }p.\endaligned\]
The partial order $\leq$ on $X$ is given by:
\[ p\leq q\Leftrightarrow C_p\subseteq C_q \Leftrightarrow U_p\supseteq U_q\] and then $U_p=\{ q\in X: q\geq p\}$ and $C_p=\{ q\in X: q\leq p\}$.

$X$  can also be regarded as a category, where the objects are the points of $X$ and the morphisms are
\[\Hom_X(p,q)=\left\{\aligned &\text{one element, denoted by } 1_{pq}, \text{ if }p\leq q\\ &\emptyset,\text{ otherwise}.\endaligned  \right.\]
The minimal elements of $X$ (with respect to the partial order) are precisely the closed points of $X$. The maximal elements of $X$ are the open points of $X$ and they will be called {\em generic points} of $X$, since their closures are the irreducible components of $X$.}
\end{cosa}

\begin{cosa} {\rm We shall denote by $\Shv_k(X)$ the category of sheaves of $k$-modules on $X$. It is equivalent to the category of functors from $X$ to the category of $k$-modules. We shall follow the notations of \cite{Sanchoetal}.

The constant sheaf $k$ on $X$ is still denoted by $k$. 

For any locally closed subset $S$ of $X$ and any sheaf $F$ on $X$, we denote by $F_S$ the sheaf $F$ supported on $S$, i.e., it is the unique sheaf on $X$ such that ${F_S}_{\vert S}=F_{\vert S}$ and ${F_S}_{\vert X-S}=0$. To avoid confusions, for each $p\in X$, $F_p$ denotes the stalk of $F$ at $p$ and $F_{\{p\}}$ denotes the sheaf $F$ supported on $p$. For any closed subset $K$ of $X$ one has an exact sequence
\[ 0\to F_{X-K}\to F\to F_K\to 0.\]

For any $k$-module $E$ and any sheaf $F$ on $X$, we denote by $E\otimes_kF$  the sheaf given by $(E\otimes_kF)_p=E\otimes_k F_p$; in other words $E\otimes_kF$ is the tensor product of the constant sheaf $E$ and $F$.

For any $p\geq q$ one has an inclusion $i_{pq}\colon k_{U_p}\hookrightarrow k_{U_q}$ and an epimorphism 
$\overline i_{pq}\colon k_{C_p}\to k_{C_q}$; for any $p,q\in X$ one has
\begin{equation}\label{i_pq} \Hom_X(k_{U_p},k_{U_q})=\left\{ \aligned k\cdot &i_{pq},\text{ if }p\geq q\\ \,\, &0, \text{ otherwise} \endaligned  \right.,\qquad \Hom_X(k_{C_p},k_{C_q})=\left\{ \aligned k\cdot &\overline i_{pq},\text{ if }p\geq q\\ \,\, &0, \text{ otherwise} \endaligned  \right.
\end{equation} because $\Hom_X(k_{U_p},F)=F_p$ and $\Hom_X(F,k_{C_p})=\Hom_k(F_p,k)$ for any sheaf $F$.

We say that a sheaf $F$ is {\em finitely generated} if $F_p$ is a finitely generated $k$-module for any $p\in X$ and we shall denote by $\Shv_{k-{\rm fg}}(X)$ the category of finitely generated sheaves on $X$. We say that $F$ is {\em flat} if $F\otimes_k \underline{\quad}\,$ is exact; equivalently, if $F_p$ is a flat $k$-module for any $p\in X$.}
\end{cosa}

\begin{cosa}{\rm {\em Standard resolutions}. For any sheaf $F$ on $X$ one has the standard resolution
\[ 0\to F\to C^0F\to C^1F\to\cdots\to C^n F\to 0,\qquad (n=\dim X)\] where $$C^iF=\prod_{p_0<\dots <p_i} F_{p_i}\otimes_k k_{C_{p_0}}$$ and the standard resolution
\[ 0\to C_n F\to\cdots\to C_1F\to C_0F\to F\to 0\] where 
\[ C_iF=\bigoplus_{p_0<\dots <p_i} F_{p_0}\otimes _k k_{U_{p_i}}.\]}

\begin{prop}\label{augmented-complex} Let $E$ be a $k$-module, $p\in X$ and $F=E\otimes_k k_{U_p}$. The augmented complex $0\to C_n F\to\cdots\to C_1F\to C_0F\to F\to 0$ is homotopically trivial. 

\begin{proof} If $F=k_{U_p}$, then $F$ is  projective and $C_iF=\bigoplus_{p\leq p_0<\cdots< p_i}k_{U_{p_i}}$ is also projective, so the result follows. Since $C_\pun(E\otimes_k k_{U_p})=E\otimes_kC_\pun k_{U_p}$, we conclude.
\end{proof}
%

\end{prop}

\end{cosa}

\begin{cosa} {\rm  {\em Koszul complexes}. Let $K$ be a closed subset of $X$ and $p_1,\dots, p_r$ the closed points of $X-K$, so $X-K=U_{p_1}\cup\dots\cup U_{p_r}$. Let us denote $L=k_{U_{p_1}}\oplus \cdots\oplus k_{U_{p_r}}$ and $\omega_K\colon L\to k$ the natural morphism, with image $k_{X-K}$. The {\em Koszul complex of $K$}, $\Kos_\pun(k_K)$, is defined by: \[\Kos_i(k_K) :={\bigwedge}^i_k L\] and the differential is the inner contraction with $\omega_K$. In other words, it is the tensor product complex
\[\Kos_\pun(k_K)=(k_{U_{p_1}}\to k)\otimes\cdots \otimes (k_{U_{p_r}}\to k)\] where each $k_{U_{p_i}}\to k$ has homological degrees $1,0$. Hence, $\Kos_\pun(k_K)$  is a (left) resolution of $k_K$.  }

\end{cosa}

\begin{cosa}{\rm  {\em Homology and cohomology}. We shall denote by $D_k(X)$ (or $D(X)$, if $k$ is understood) the derived category of complexes of sheaves of $k$-modules on $X$. For any $F\in D(X)$, we shall denote by $\RR\Gamma(X,F)$ the right derived functor  of sections of $F$; analogously, we shall denote by $\LL(X,F)$ the left derived functor of cosections of $F$ (see \cite{Sanchoetal}), and
\[ H^i(X,F)=H^i[\RR\Gamma(X,F)],\qquad H_i(X,F)=H_i[\LL(X,F)]\] the cohomology and homology $k$-modules of $F$. These may be computed with the standard resolutions, i.e.,
\[ \RR\Gamma(X,F)\simeq \Gamma(X,C^\pun F),\quad  \LL(X,F)\simeq \L(X,C_\pun F).\]
Homology and cohomology of the constant sheaf $k$ are mutually dual: $$\RR\Gamma(X,k)=\RR\Hom^\pun_k(\LL(X,k),k) 
\text{ and }\LL(X,k)=\RR\Hom^\pun_k(\RR\Gamma(X,k),k).$$ 

We shall denote   $\LL_\text{\rm red}(X,k):=\operatorname{Cone}(\phi)[-1]$, where $\phi$ is the natural morphism of complexes $\phi\colon \LL(X,k)\to k$. Then
\[ \widetilde H_i(X,k):=H_i[\LL_\text{\rm red}(X,k)]\] are the reduced homology $k$-modules of $X$. Notice that, if $X$ is empty, then $\widetilde H_{-1}(X,k)=k$ and $\widetilde H_i(X,k)=0$ for any $i\neq -1$.

The global dualizing complex of $X$ will be denoted by $D_X$; it is the object $D_X\in D(X)$ satisfying
\[ \RR\Hom_X^\pun(F,k)=\RR\Hom^\pun_k(\RR\Gamma(X,F),k)\] for any $F\in D(X)$.

Finally, we shall denote by $D_c(X)$ the full subcategory of $D(X)$ consisting of those complexes   $F\in D(X)$ such that $H^i(F)$ is finitely generated for any $i\in\ZZ$.}
\end{cosa}

\subsection{Affine and projective finite spaces}$\,$

\begin{cosa}{\rm  Let $\Delta_{n}=\{ 1,\dots,n\}$ be the finite set with $n$ elements. 

\begin{defn}{\rm We shall denote $$\A^n_{\FF_1}:=\mathcal{P}(\Delta_{n})$$ the set of subsets of $\Delta_{n}$, which is a finite topological space (a finite poset) with the topology given by the partial order defined by inclusion: $p\leq q\Leftrightarrow p\subseteq q$. }
\end{defn} We shall denote:

\begin{enumerate}

\item[$\bullet$] $\0\in\A^{n}_{\FF_1}$ the element representing the empty subset of $\Delta_{n}$. It is the unique closed point of $\A^{n}_{\FF_1}$.

\item[$\bullet$] $\un\in\A^{n}_{\FF_1}$ the generic point of $\A^n_{\FF_1}$, i.e., the element representing the total subset of $\Delta_{n}$. It is the unique open point of $\A^{n}_{\FF_1}$.

\item[$\bullet$] For each $i\in\Delta_{n}$, we shall denote by $U_i$ the open subset of $\A^{n}_{\FF_1}$ whose elements are the subsets of $\Delta_{n}$ containing $i$; in other words $U_i=U_{\{i\}}$ (notation of \ref{cosa1}). 

\item[$\bullet$] We shall denote   $H_i:=\A^{n}_{\FF_1}- U_i$ the hyperplane of $\A^{n}_{\FF_1}$ whose elements are the subsets of $\Delta_{n}$ contained in $\Delta_{n}-\{ i\}$. 

\item[$\bullet$] For any   $p\in \A^{n}_{\FF_1}$, $\widehat p$ denotes the complement subset: $\widehat p=\Delta_n - p$. 
\end{enumerate}

The union and intersection of subsets of $\Delta_n$ define continuous maps
\[ \aligned \A^n_{\FF_1}\times\A^n_{\FF_1} &\overset +\longrightarrow \A^n_{\FF_1}\\ (p,q)&\mapsto p+q:=p\cup q\endaligned,\qquad \aligned \A^n_{\FF_1}\times\A^n_{\FF_1} &\overset \cdot\longrightarrow \A^n_{\FF_1}\\ (p,q)&\mapsto p\cdot q:=p\cap q\endaligned.\] For any $p,q\in\A^n_{\FF_1}$ one has
\[ U_p\cap U_q=U_{p+q}\quad,\quad C_p\cap C_q = C_{p\cdot q}.\]

\begin{defn}{\rm  We shall denote  $$\PP^n_{\FF_1} :=\A^{n+1}_{\FF_1} \negmedspace -\negmedspace \0$$ the set of non-empty subsets of $\Delta_{n+1}=\{ 0,1,\dots,n\}$, which is an open subset of $\A^{n+1}_{\FF_1}$. }\end{defn}
Notice that $U_0,\dots, U_n$ is an open covering of $\PP^n_{\FF_1}$. 
When working on $\PP^n_{\FF_1}$, we shall still denote by $H_i$ the hyperplane of $\PP^n_{\FF_1}$ which is the complement    of $U_i$; that is, it is the restriction of the hyperplane $H_i$ of $\A^{n+1}_{\FF_1}$ to $\PP^n_{\FF_1}$. On $\PP^n_{\FF_1}$, $H_i$ is homeomorphic to $\PP^{n-1}_{\FF_1}$.}
\end{cosa}

\begin{cosa}{\rm The universal properties of $\A^n_{\FF_1}$ and $\PP^n_{\FF_1}$ are the following. Let $T$ be any topological space. Giving a continuous map $f\colon T\to \A^{n}_{\FF_1}$ is equivalent to giving $n$ closed subsets $C_1,\dots, C_n$ of $T$; explicitly, the map
\begin{equation}\label{univpropAn} \aligned \Hom_{\rm cont}(T,\A^{n}_{\FF_1})&\longrightarrow \{ (C_1,\dots ,C_n): C_i \text{ a closed subset of } T\}
\\ f &\longmapsto (f^{-1}(H_1),\dots , f^{-1}(H_n))
\endaligned\end{equation}
is bijective, and the inverse sends a collection $(C_1,\dots ,C_n)$ of closed subsets of $T$ to the continuous map $f\colon T\to  \A^{n}_{\FF_1}$ defined by
\[ f(t)=\{ i\in\Delta_{n}: t\notin C_i\}.\]

The bijection \eqref{univpropAn} induces the following one
\[\aligned \Hom_{\rm cont}(T,\PP^{n}_{\FF_1} )&\longrightarrow \{ (C_0,\dots ,C_n): C_0\cap\dots \cap C_n=\emptyset \}
\\ f &\longmapsto (f^{-1}(H_0),\dots , f^{-1}(H_n))
\endaligned\]}\end{cosa}

\begin{defn}{\rm We shall denote $\HH^n:=\PP^{n+1}_{\FF_1}\negmedspace -\negmedspace \un=\A^{n+2}_{\FF_1}-\{\0,\un\}$ the set of non-trivial subsets of $\Delta_{n+2}$. It is the closed subset of $\PP^{n+1}_{\FF_1}$ consisting in the union of its hyperplanes.}
\end{defn}

\begin{cosa}\label{basiccohomologies}{\rm The cohomology of these spaces for the constant sheaf is:
\[ H^\pun(\A^n_{\FF_1},k)=k,\quad H^\pun(\PP^n_{\FF_1},k)=k, \quad H^\pun(\HH^n,k)=k\oplus k[-n].\] 
%
and their global dualizing complexes  are:
\[ D_{\A^n_{\FF_1}}=k_{\{\0\}},\quad D_{\PP^n_{\FF_1}}=k_{\{\un\}}[n],\quad D_{\HH^n}=k[n].\] }
\end{cosa}

For a longer justification of the names (and notations) {\em affine} and {\em projective} given to the spaces $\A^n_{\FF_1}$ and $\PP^n_{\FF_1}$ see \cite{SS}.

\section{Yoneda Extension}\label{Yoneda}

Let $X$ be a finite topological space and $X^{\text{\rm op}}$ its dual. Let us consider the functors 
\[  \iota\colon  X^{\text{\rm op}}\to\Shv_k(X),\qquad
\overline{\iota}\colon  X^{\text{\rm op}}\to\Shv_k(X) \] defined by (see \eqref{i_pq})
\[\aligned  \iota(p)&=k_{U_p}, \quad \iota(1_{pq}) =i_{pq}\colon k_{U_p}\hookrightarrow k_{U_q}\\
\overline{\iota}(p)&=k_{C_p}, \quad \overline{\iota}(1_{pq}) =\overline{i}_{pq}\colon k_{C_p}\to k_{C_q}\endaligned, \quad \text{ for any }p\geq q.\]

The category $\Shv_k(X)$ is an example of what we shall call a $k$-linear category:

\begin{defn}{\rm By a {\em $k$-linear category} we mean an abelian category $\C$ such that:

(a) $\Hom_\C(A,B)$ is a $k$-module for any objects $A,B\in\C$, and the composition
\[ \Hom_\C(A,B)\times \Hom_\C(B,C)\to\Hom_\C(A,C)\] is bilinear, for any $A,B,C$.

(b) For any $k$-module $E$ and any object $A$ of $\C$, there is an object $E\otimes_k A$ satisfying: for any $B\in\C$ one has a $k$-linear isomorphism
\[ \Hom_\C(E\otimes_k A,B)\simeq \Hom_k(E,\Hom_\C(A,B))\] which is functorial on $B$.

A $k$-linear functor $F\colon \C\to\C'$ between $k$-linear categories is an additive functor such that $\Hom_\C(A,V)\to \Hom_{\C'}(F(A),F(B)$ is $k$-linear and such that  the natural morphism 
\[ E\otimes_k F(A)\to F(E\otimes_k A)\] is an isomorphism.

An object $A\in\C$ is called {\em $k$-flat} if the functor $\underbar{\quad}\otimes_k A\colon \C_{k-\text{\rm mod}}\to\C$ is exact.}
\end{defn}

\begin{rems}{\em (1) If $\C$ is cocomplete, then condition (b) is unnecessary, it is always satisfied (assuming (a)).

(2) If we replace (b) by the same condition but only for finitely generated $k$-modules $E$, then we say that $\C$ is a  (fg)-$k$-linear category. If $k$ is noetherian, then this is automatically satisfied (assuming (a)).}
\end{rems}

\begin{ex} {\rm Let $S$ be a scheme over $k$ and $\pi\colon S\to\Spec k$ the structure morphism. The category $\Qcoh(S)$ of quasi-coherent $\OO_S$-modules is $k$-linear by defining $$E\otimes_k\M:=\pi^*E\otimes_{\OO_S}\M,$$  for any $k$-module $E$ and any quasi-coherent module $\M$. Analogously, the category $\Coh(S)$ of coherent $\OO_S$-modules is a (fg)-$k$-linear category (assuming $k$  noetherian). A quasi-coherent module $\M$ is $k$-flat if and only if it is flat over $k$ in the usual sense.}
\end{ex}

The following theorem shows that $\Shv_k(X)$ is ``the $k$-linear envelope'' of $X^{\text{\rm op}}$. It is a slight modification of the well known Yoneda Extension Theorem \cite[Thm. 2.3.3, Prop. 2.7.1]{KS}.

\begin{thm}\label{extension} Let $\C$ be a $k$-linear category and let \[\Phi_X\colon  X^{\text{\rm op}}\to\C\] be a functor (i.e., a $\C$-valued sheaf on $X^{\text{\rm op}}$). There exists a  unique (upto a unique natural isomorphism) $k$-linear  right exact (resp. left exact) functor
\[  \Phi\colon  \Shv_k(X)\to\C \qquad (\text{\rm resp. } \overline\Phi\colon  \Shv_k(X)\to\C )\]  such that  $\Phi\circ\iota=\Phi_X$ (resp., such that  $\overline\Phi\circ\overline\iota=\Phi_X$). This yields equivalences
\[ \aligned \left\{\aligned & \text{Right exact } k\text{-linear}\\ &\text{functors } \Shv_k(X)\to\C \endaligned  \right\}
\simeq  
&\left\{  \C\text{-valued  sheaves on } X^{\text{\rm op}}\right\} 
\simeq \left\{\aligned & \text{Left exact } k\text{-linear}\\ &\text{functors } \Shv_k(X)\to\C \endaligned  \right\} \\ \Phi\quad &\longleftarrow  \quad\quad \quad\Phi_X\quad\quad\quad\longrightarrow\quad\overline\Phi
\\ \Phi\quad &\longrightarrow    \quad \Phi\circ\iota,\quad \overline\Phi\circ\overline\iota  \quad\longleftarrow\quad\overline\Phi
\endaligned
\]
\end{thm}

\begin{proof} Let $F$ be a sheaf on $X$, with restriction morphisms $r_{xy}\colon F_x\to F_y$ for each $x\leq y$. Let us consider the exact sequence
\[ C_1F\overset{\di}\to C_0F\to F\to 0 \] where:
\[\aligned C_0F:=&\bigoplus_{x\in X} F_x\otimes_k k_{U_x}\\ C_1F:= &\bigoplus_{x<y} F_x\otimes_k k_{U_y}
\endaligned\] and    $\di \colon C_1F\to C_0F$ is the morphism induced by the natural morphism
\[ F_x\otimes_k k_{U_y}\overset{(r_{xy}\otimes 1, -1\otimes i_{yx})}\longrightarrow F_y\otimes_k k_{U_y}\oplus  F_x\otimes_k k_{U_x}.\] 


 Let us define the functors $C_0\Phi, C_1\Phi\colon\Shv_k(X)\to\C$ by
\[ (C_0\Phi)(F):= \bigoplus_{x\in X} F_x\otimes_k \Phi_X(x) , \quad (C_1\Phi)(F):= \bigoplus_{x<y} F_x\otimes_k \Phi_X(y)\] and let $\Phi(\di)\colon (C_0\Phi)(F)\to (C_1\Phi)(F)$ be the morphism induced by the morphisms
\[ F_x\otimes_k \Phi_X(y)\overset{(r_{xy}\otimes 1, -1\otimes \Phi_X(1_{yx}))}\longrightarrow F_y\otimes_k \Phi_X (y)\oplus  F_x\otimes_k \Phi_X(x).\] 
Let us define $\Phi(F):=\Coker(\Phi(\di))$. Since $C_0\Phi$ and $C_1\Phi$ are $k$-linear and right exact, so is $\Phi$. Let us see that $\Phi\circ\iota=\Phi_X$, i.e. $\Phi(k_{U_p})=\Phi_X(p)$. For each $p\in X$, one has
\[ (C_0\Phi)(k_{U_p})=\bigoplus_{x\geq p}\Phi_X(x),\qquad (C_1\Phi)(k_{U_p})=\bigoplus_{y>x\geq p}\Phi_X(y).\] The natural projection $\bigoplus_{x\geq p}\Phi_X(x)\to \Phi_X(p)$ gives a morphism $(C_0\Phi)(k_{U_p})\to \Phi_X(p)$ and we leave the reader to check that the sequence
\[ \bigoplus_{y>x\geq p}\Phi_X(y)\overset{\Phi(d)}\longrightarrow \bigoplus_{x\geq p}\Phi_X(x)\to \Phi_X(p)\to 0\] is exact (in fact, it is homotopically trivial, whose homotopy operator is just ``applying $\Phi$'' to the homotopy operator of the augmented complex $C_\pun k_{U_p}\to k_{U_p}$). This says that $\Phi(k_{U_p})=\Phi_X(p)$ (here, and in the sequel, equality means a natural isomorphism).

Finally, let us see the unicity of $\Phi$. Let $\Psi\colon \Shv_k(X)\to\C$ be a $k$-linear right exact functor such that $\Psi\circ\iota=\Phi_X$; then $\Psi(C_0F)=(C_0\Phi)(F)$,  $\Psi(C_1F)=(C_1\Phi)(F)$ and $\Psi(d)=\Phi(d)$ because $\Psi$ is $k$-linear and $\Psi\circ\iota=\Phi_X$. We conclude that $\Psi(F)=\Phi(F)$ because $\Psi$ is right exact.

The construction of $\overline\Phi$ is completely analogous, starting from the exact sequence $0\to F\to C^0F\to C^1F$, where $C^0F=\prod_{x\in X} F_x\otimes_k k_{C_x}$, $C^1F:=  \prod_{x<y} F_y\otimes_k k_{C_x}$.
\end{proof}

\begin{rem}{\rm If $\C$ is a (fg)-$k$-linear category, the same proof shows that any func\-tor $\Phi_X\colon X^{\text{op}}\to\C$ extends in a unique way to a right-exact $k$-linear functor $\Phi\colon \Shv_{k-\text{\rm fg}}(X)\to\C$, where $\Shv_{k-\text{fg}}(X)$ denotes the category of sheaves of finitely generated $k$-modules on $X$.}
\end{rem}

\section{Algebraic morphisms}\label{St-R-functor} In this section we shall apply the results of the previous section to  the finite topological space $\A^n_{\FF_1}$ and the abelian category $\C=\Qcoh(S)$ for some scheme $S$ over $k$.

Recall that on $\A^n_{\FF_1}$ one has $$U_\0=\A^n_{\FF_1},\qquad U_p=\underset{i\in p}\bigcap\, U_i, \text{ for any }p\neq \0 $$ and

$$C_\un=\A^n_{\FF_1},\qquad C_p=\underset{i\notin p}\bigcap\, H_i, \text{ for any }p\neq \un. $$

Hence
\[ k=k_{U_\0},\quad k_{U_p}=\underset{i\in p}{\bigotimes }\, k_{U_i},\quad k_{C_p}=\underset{i\notin p}{\bigotimes }\, k_{H_i} .\]

\begin{cosa}{\rm Let $S$ be a scheme, $\Lc_1,\dots,\Lc_n$ quasi-coherent modules on $S$ and, for each $i=1,\dots ,n$, let $s_i\colon \Lc_i\to\OO_S$ be a morphism of $\OO_S$-modules. For brevity, we shall denote $(\Lc,s)$ these data.

\begin{defn}{\rm Let us denote $\I_i$ the image of $s_i$, which is a quasi-coherent ideal, and $D_i$ the closed subscheme defined by $\I_i$. We have then a continuous map
 \[ f \colon S\longrightarrow\A^n_{\FF_1}\] uniquely determined by $f ^{-1}(H_i)=D_i$ (see \eqref{univpropAn})}.
 \end{defn}
 
\begin{defn}{\rm For each $p\in\A^n_{\FF_1}$, let us denote
 \[\Lc_\0=\OO_S,\quad \Lc_p=\underset{i\in p}\otimes \Lc_i, \text{ for }p\neq\0\] and $s_p:=\underset{i\in p}\otimes s_i\colon \Lc_p\to\OO_S$ ($s_p=\id$ if $p=\0$). For each $p\leq q$ one has $\Lc_q=\Lc_p\otimes_{\OO_S}\Lc_{q-p}$ and then a morphism
 \[ 1\otimes s_{q-p}\colon \Lc_q\to \Lc_p.\] We have then a functor
 \[ \aligned   (\A^n_{\FF_1})^{\rm op}&\longrightarrow\Qcoh(S)\\ p &\longmapsto \Lc_p \\ 1_{pq}&\longmapsto 1\otimes s_{q-p}\endaligned\]
 whose $k$-linear right-exact extension will be denoted by
 \[ f^\stella\colon \Shv_k(\A^n_{\FF_1})\longrightarrow \Qcoh(S).\]}
 \end{defn}
 
\begin{defn}{\rm For each $p\in\A^n_{\FF_1}$, let us denote by $\I_p$ the image of $s_p\colon \Lc_p\to\OO_S$. If $p=\{ i_1,\dots,i_r\}$, then
 \[\I_p=\I_{i_1}\cdot \cdots\cdot\I_{i_r}.\]
Now, for any closed subset $K$ of $\A^n_{\FF_1}$, let us denote $$\I_K:=\underset{p\notin K}\sum \I_p$$ and let $S_K$ be the closed subscheme defined by $\I_K$. We shall denote $\OO_{S_K}:=\OO_S/\I_K$.}\end{defn} }
\end{cosa}

\begin{prop} One has:\begin{enumerate}\item  $f^{-1}(K)=S_K$ (topologically).
\item $f^\stella(k_K)=\OO_{S_K}$.
\end{enumerate}
\end{prop}

\begin{proof} (1) $S_K=(\I_K)_0=\underset{p\notin K}\bigcap (\I_p)_0$ and $K=\underset{p\notin K}\bigcap (\A^n_{\FF_1}- U_p)$, so we are reduced to see that $(\I_p)_0= f^{-1}(\A^n_{\FF_1}- U_p)$. Since $(\I_p)_0= \underset{i\in p}\bigcup D_i$ and $\A^n_{\FF_1}- U_p=\underset{i\in p}\bigcup H_i$, one concludes because $f^{-1}(H_i)=D_i$.

(2) It suffices to apply $f^\stella$ to the exact sequence $\underset{p\notin K}\bigoplus k_{U_p}\to k\to k_K\to 0.$
\end{proof}

\begin{rem}{\rm In particular, for each $p\in\A^n_{\FF_1}$: $$\underset{i\notin p}\bigcap D_i= f^{-1}(C_p)=S_{C_p}\qquad \text{ and }\qquad  f^\stella(k_{C_p})=\OO_{S_{C_p}}=\OO_{\underset{i\notin p}\bigcap D_i}.$$}
\end{rem}

\begin{rem}{\rm Another data $(\Lc',s')$   on $S$ is said to be equivalent to $(\Lc,s)$, denoted by $(\Lc',s')\sim (\Lc,s)$, if there are isomorphisms $\phi_i\colon \Lc_i\overset\sim\to\Lc'_i$ such that $s'_i\circ\phi_i= s_i$. In this case, one has that $f =f'$, where $f'$ is the continuous map associated to $(\Lc',s')$, and  a natural isomorphism $f^\stella \simeq {f'}^\stella $. Hence, we shall make no distinction between two equivalent data, i.e., a data $(\Lc,s)$ is to be understood as an equivalence class of data.}
\end{rem}

\subsubsection{Algebraic morphisms} Let us assume now that the $\Lc_i$ are invertible $\OO_S$-modules and then $s_i\colon \Lc_i\to\OO_S$ is a global section of $\Lc_i^{-1}$. In this case  we say that $(\Lc,s)$ is a data of invertibles.

\begin{defn} {\rm Let $S$ be a $k$-scheme,  $(\Lc,s)$ a data of invertibles and
\[ f\colon S\to\A^n_{\FF_1}\quad,\quad f^\stella\colon \Shv_k(\A^n_{\FF_1})\to\Qcoh(S)\] the continuous map and the right-exact $k$-linear functor associated to them. We say that the pair $(f\colon S\to\A^n_{\FF_1} ,f^\stella)$ is the {\em algebraic $k$-morphism}  associated to $(\Lc,s)$. Usually, we shall use the abbreviated notation $f\colon S\to\A^n_{\FF_1}$ for such a morphism.

If the morphisms $s_i\colon \Lc_i\to\OO_S$ are injective, then $D_i$ are effective Cartier divisors and $\Lc_i\overset\sim\to \I_i=\OO_S(-D_i)$. In this case, the data $(\Lc,s)$ amount to a collection $D=(D_1,\dots, D_n)$ of effective Cartier divisors on $S$, and we shall denote
\[ f_D\colon S\longrightarrow\A^n_{\FF_1}\] the  algebraic $k$-morphism defined by them.  }
%
\end{defn}

\begin{rem}{\rm Algebraic morphisms appear in a natural way. For example, let $S$ be a regular   scheme  and let $U_1,\dots, U_n$ be an affine open covering of $S$. The complements $D_i=X-U_i$ are effective Cartier divisors, so they define an algebraic morphism $f\colon S\to \A^n_{\FF_1}$ (in fact, to $\PP^{n-1}_{\FF_1}$). More generally, if $Y\subseteq S$ is a closed subset and $U_1,\dots, U_n$ is an affine open covering of $X-Y$, then the Cartier divisors $D_i=X-U_i$ define an algebraic morphism $f\colon S\to \A^n_{\FF_1}$ such that $f^{-1}(\0)=Y$. See \cite{Sa} for a more systematic study of the interplay between schemes and finite spaces and \cite{SaTo2} as an example of how to reduce Grothendieck duality of schemes to that of ringed finite spaces.}
\end{rem}

\begin{rem}\label{composition}{\rm Let $f\colon S\to\A^n_{\FF_1}$ be the algebraic $k$-morphism associated to $(\Lc,s)$ and $g\colon S'\to S$ a morphism of $k$-schemes. We define $f\circ g\colon S'\to \A^n_{\FF_1}$ as the algebraic $k$-morphism associated to the data $(g^*\Lc,g^*(s))$. The associated continuous map coincides with $g\circ f$ and  
\[  (g\circ f)^\stella= g^*\circ f^\stella.\]}
\end{rem}

\begin{exs}{\rm (Classic  Stanley--Reisner correspondence) (1) Let us take $$S=\A^n_k:=\Spec k[x_1,\dots ,x_n]$$ the $n$-dimensional affine scheme over $k$; for each $i=1,\dots, n$,  let $D_i$ be the hyperplane defined by $x_i=0$ and $\pi_D\colon \A^n_k\to\A^n_{\FF_1}$ the  algebraic $k$-morphism defined by $D_1,\dots,D_n$. The   continuous map $\pi \colon \A^n_k\to\A^n_{\FF_1}$ is surjective and the functor 
\[\pi^\stella \colon \Shv_k(\A^n_{\FF_1})\to\Qcoh(\A^n_k)\] is exact. Moreover, $\pi^\stella (F)=0$ if and only if $F=0$ (for a finitely generated $F$). For any closed subset $K$ of $\A^n_{\FF_1}$, the closed subscheme $\pi^{-1}(K)=S_K\subseteq \A^n_k$ is precisely the Stanley--Reisner scheme associated to the simplicial complex $K$.

(2) Now, let $$S=\PP^n_k:=\Proj k[x_0,\dots ,x_n]$$ be the $n$-dimensional projective scheme over $k$; for each $i=0,\dots ,n$, let $D_i$ be the hyperplane with homogeneous equation $x_i=0$ and $\pi_D\colon \PP^n_k\to\A^{n+1}_{\FF_1}$ the algebraic $k$-morphism defined by $D_0,\dots,D_n$.  The image of  $\pi \colon\PP^n_k\to \A^{n+1}_{\FF_1}$ is $\PP^n_{\FF_1}$, and the functor $\pi^\stella\colon \Shv_k(\A^{n+1}_{\FF_1})\to\Qcoh(\PP^n_k)$ induces a functor (see \ref{abierto-cerrado}, (a), below) $$\pi^\stella_{\PP^n_{\FF_1}} \colon\Shv_k(\PP^n_{\FF_1})\to \Qcoh(\PP^n_k)$$ which  is exact. Moreover,  $\pi_{\PP^n}^\stella (F)=0$ if and only if $F=0$ (for a  finitely generated sheaf $F$). For any closed subset $K$ of $\A^{n+1}_{\FF_1}$ the closed subscheme $\pi^{-1}(K)\subseteq \PP^n_k$ (with structure sheaf $\OO_{\pi^{-1}(K)}=\pi^\stella_{\PP^n_{\FF_1}} (k_{K-\0})$) is the projectivization of  the Stanley--Reisner scheme associated to the simplicial complex $K$.

}

\begin{cosa}[\bf Direct image]\label{directimage} {\rm  Let $f\colon S\to\A^n_{\FF_1}$ be the algebraic $k$-morphism associated to $(\Lc,s)$. By construction, $f^\stella$ is right exact and commutes with direct limits. Hence, it has a right adjoint
\[ f_\stella\colon \Qcoh(S)\longrightarrow \Shv_k(\A^n_{\FF_1})\] whose explicit description is the following: for each $\M\in\Qcoh(S)$ and each $p\in\A^n_{\FF_1}$ one has \[ (f_\stella\M)_p=\Gamma(S,\M\otimes \Lc_p^{-1}).\] Indeed, $(f_\stella\M)_p=\Hom_{\A^n_{\FF_1}}(k_{U_p},f_\stella\M) = \Hom_{S}(f^\stella k_{U_p}, \M) = \Gamma(S,\M\otimes \Lc_p^{-1})$. 

If $S$ is an affine scheme, then $f_\stella$ is exact. 
}
\end{cosa}

\begin{cosa}[\bf Products]\label{product}{\rm Let $f_1\colon S_1\to \A^{n_1}_{\FF_1}$, $f_2\colon S_2\to \A^{n_2}_{\FF_1}$ be two algebraic $k$-morphisms, associated, respectively, to $(\Lc_1,s_1)$ and $(\Lc_2,s_2)$, and let us consider the product scheme $S_1\times_k S_2$.\medskip

For any quasicoherent-modules $\M_1\in\Qcoh(S_1)$, $ \M_2\in \Qcoh(S_2)$, let us denote
\[ \M_1\boxtimes\M_2=p_1^*\M_1\otimes_{\OO_{S_1\times_kS_2}}p_2^*\M_2,\qquad p_i\colon S_1\times_kS_2 \to S_i\]
 and, analogously,  for any sheaves $F_1\in\Shv_k(\A^{n_1}_{\FF_1})$, $F_2\in\Shv_k(\A^{n_2}_{\FF_1})$
\[ F_1\boxtimes F_2=\pi_1^{-1}F_1\otimes_k\pi_2^{-1}F_2,\qquad \pi_i\colon \A^{n_1}_{\FF_1}\times \A^{n_2}_{\FF_1} \to \A^{n_i}_{\FF_1}.\] The data $(p_1^*\Lc_1,p_1^*s_1)$, together with $(p_2^*\Lc_2,p_2^*s_2)$, define an algebraic morphism $$f_1\times f_2\colon S_1\times S_2 \to\A^n_{\FF_1},\qquad  n=n_1+n_2.$$ 
Then one has:
\[ (f_1\times f_2)^\stella (F_1\boxtimes F_2)= (f_1^\stella F_1)\boxtimes (f_2^\stella F_2).\] Moreover, if $S_1$ and $S_2$ are flat over $k$, then
\[ (f_1\times f_2)_\stella (\M_1\boxtimes \M_2)= ({f_1}_\stella \M_1)\boxtimes ({f_2}_\stella \M_2).\] Indeed, for each $(p,q)\in \A^{n_1}_{\FF_1}\times \A^{n_2}_{\FF_1}$, one has
\[\aligned [(f_1\times f_2)_\stella (\M_1\boxtimes \M_2)]_{(p,q)} &= \Hom_{\A^n_{\FF_1}}( k_{U_{(p,q)}}, (f_1\times f_2)_\stella (\M_1\boxtimes \M_2) ) 
\\ &= \Hom_{\A^n_{\FF_1}}( k_{U_p}\boxtimes k_{(U_q}, (f_1\times f_2)_\stella (\M_1\boxtimes \M_2) )
\\ & =  \Hom_{S_1\times_k S_2}((f_1\times f_2)^\stella  (k_{U_p}\boxtimes k_{U_q}),  \M_1\boxtimes \M_2  ) 
\\ &= \Hom_{S_1\times_k S_2}((f_1^\stella k_{U_p}) \boxtimes (f_2^\stella k_{U_q}),  \M_1\boxtimes \M_2  ) 
\\ &\overset{\text{by flatness}}= \Hom_{S_1 }( f_1^\stella k_{U_p},    \M_1   )  \otimes_k 
\Hom_{S_2 }( f_2^\stella k_{U_q},    \M_q   ) 
\\ &=({f_1}_\stella \M_1)_p \otimes_k ({f_2}_\stella \M_2)_q = [({f_1}_\stella \M_1)\boxtimes ({f_2}_\stella \M_2)]_{(p,q)}.\endaligned\]
}\end{cosa}

\begin{cosa} {\rm Let $f\colon X\to \A^n_{\FF_1}$ be an algebraic $k$-morphism and $F,G\in\Shv_k(\A^n_{\FF_1})$. Let us see that there is a natural morphism
\begin{equation}\label{fstellatensor} (f^\stella F)\otimes_{\OO_S} (f^\stella G)\longrightarrow f^\stella (F\otimes_kG).\end{equation}

For any $p\in \A^n_{\FF_1}$ one has a natural morphism  $\Psi_p\colon\Lc_p\otimes_{\OO_S}\Lc_p\to\Lc_p$, induced by $s_p\colon \Lc_p\to\OO_S$. Then, for any $p,q\in\A^n_{\FF_1}$ one has a natural morphism 
\[ \Lc_p\otimes_{\OO_S}\Lc_q\to\Lc_{p+q},\] 
defined as follows: let us denote $p'=p-(p\cdot q)$, $q'=q-(p\cdot q)$, so $\Lc_p=\Lc_{p'}\otimes \Lc_{p\cdot q}$ and $\Lc_q=\Lc_{q'}\otimes \Lc_{p\cdot q}$, and then 
\[ \Lc_p\otimes \Lc_q = \Lc_{p'}\otimes \Lc_{q'}\otimes \Lc_{p\cdot q} \otimes \Lc_{p\cdot q} \overset{1\otimes \Psi_{p\cdot q}}\longrightarrow \Lc_{p'}\otimes \Lc_{q'}\otimes \Lc_{p\cdot q}=\Lc_{p+q}.\]
 If we consider the (non-commutative) diagram
\[ \xymatrix{ {\A^n_{\FF_1}}^\text{op}\times {\A^n_{\FF_1}}^\text{op} \ar[r]^{\Phi\times\Phi\quad}\ar[d]_{+} & \Qcoh(S)\times \Qcoh(S)\ar[d]^{\otimes} \\ {\A^n_{\FF_1}}^\text{op}\ar[r]^{\Phi} &\Qcoh(S)}\] the above morphisms $\Lc_p\otimes\Lc_q\to\Lc_{p+q}$ define a morphism of functors
\[ \otimes\circ (\Phi\times\Phi)\to \Phi\circ +\] and then, by Yoneda extension, a morphism of functors
\[ \delta_S^*\circ (f\times f)^\stella \longrightarrow f^\stella\circ \delta^{-1}\] where $\delta_S\colon S\to S\times_kS$ and $\delta \colon \A^n_{\FF_1}\to \A^n_{\FF_1}\times \A^n_{\FF_1}$ are the diagonal morphisms. Applied to $F\boxtimes G$, one obtains (taking into account  \ref{product}) the desired morphism $(f^\stella F)\otimes_{\OO_S} (f^\stella G)\to f^\stella (F\otimes_kG)$.
}\end{cosa}

\begin{cosa}[\bf Restriction to open and closed subsets]\label{abierto-cerrado}{\rm  Let $S$ be a $k$-scheme and $f \colon S\to\A^n_{\FF_1}$ an algebraic $k$-morphism. 

(a) Let $U\overset i\hookrightarrow \A^n_{\FF_1}$ be an open subset, $S_U=f^{-1}(U)$ the induced open subscheme of $S$ and $i_S\colon S_U\hookrightarrow S$ the inclusion. We have then a continuous map $f_{\vert U}\colon S_U\to U$, a right exact $k$-linear functor
\[ f_U^\stella :=i_S^*\circ f^\stella \circ i_!\colon \Shv_k(U)\to\Qcoh(S_U).\]
 and  a commutative diagram
\[\xymatrix{ \Shv_k(\A^n_{\FF_1})\ar[r]^{f^\stella } \ar[d]_{i^{-1}} & \Qcoh(S)\ar[d]^{i_S^*}\\ \Shv_k(U)\ar[r]^{f_U^\stella } & \Qcoh(S_U).
} \] If $f^\stella$ is exact, so is $f_U^\stella $.\medskip

(b) 
Let $j\colon K\hookrightarrow \A^n_{\FF_1}$ be a closed subset,   $S_K$ the associated closed subscheme of $S$ and   $j_S\colon S_K\hookrightarrow S$ the closed immersion. Let $F$ be a sheaf on $\A^n_{\FF_1}$ supported on $K$ (i.e., $F=F_K=j_*F_{\vert K}$); then $f^\stella F$ is a quasicoherent $\OO_S$-module annihilated by $\I_K$. Indeed, Let $p_1,\dots p_r$ be the closed points of $\A^n_{\FF_1}-K$ and let us consider the exact sequence
\[ \bigoplus_{i=1}^r k_{U_{p_i}}\to k\to k_K\to 0.\] Since $F$ is supported at $K$, the morphism $F\to F_K$ is an isomorphism, so $F\otimes_k k_{U_{p_i}}\to F$ is null. Then it is also null the morphism
\[f^\stella F\otimes f^\stella k_{U_{p_i}} \to f^\stella(F\otimes_k k_{U_{p_i}})\to f^\stella F\] and then $f^\stella F$ is annihilated by $\I_K$.  Thus,  $f^\stella F $ is an $\OO_{S_K}$-module, and then $f^\stella F={j_S}_*j_S^{-1}f^\stella F$ and $j_S^{-1}f^\stella F$ is a quasi-coherent module on $S_K$. Thus we obtain a  right-exact  functor
$$f_{K}^\stella \colon \Shv_k(K)\to\Qcoh(S_K), \quad F\mapsto j_S^{-1}(f^\stella j_*F) $$ and a commutative  diagram
\[\xymatrix{ \Shv_k(\A^n_{\FF_1})\ar[r]^{f^\stella }   & \Qcoh(S) \\ \Shv_k(K)\ar[r]^{ f_{K}^\stella  }\ar[u]_{j_*} & \Qcoh(S_K)\ar[u]^{{j_S}_*}.
} \] If $f^\stella$ is exact, then $f_{K}^\stella$ is exact. 

\begin{rem}{\rm All the results that we shall obtain for an algebraic morhism $f\colon S\to \A^n_{\FF_1}$ may be easily generalized to $f_U$ or $f_K$.}
\end{rem}
}\end{cosa}

\begin{cosa}[\bf Derived functors] {\rm We shall denote $$\LL f^\stella\colon D(\A^n_{\FF_1})\longrightarrow D(S)$$ the left derived functor of $f^\stella$ (notice that $\A^n_{\FF_1}$ has enough projectives) and
\[ \RR f_\stella\colon D(S) \longrightarrow  D(\A^n_{\FF_1})\] the right derived functor of $f_\stella$. They are $k$-linear:
\[ \LL f^\stella (E\overset\LL \otimes_kF)=E\overset\LL\otimes_k \LL f^\stella F, \qquad \RR f_\stella (E\overset\LL \otimes_k \M)=E\overset\LL\otimes_k \RR f_\stella F\] for any $E\in D(k), F\in D(\A^n_{\FF_1}),\M\in D(S)$, and mutually adjoint:
\[ \RR\Hom_S^\pun(\LL f^\stella F,\M)=\RR\Hom_{\A^n_{\FF_1}}^\pun(F,\RR f_\stella\M).\]

We leave the reader to give a natural morphism (extending \ref{fstellatensor}) $$(\LL f^\stella F)\overset\LL\otimes (\LL f^\stella G) \to \LL f^\stella (F\overset\LL\otimes G)$$ and then a morphism
\[ \LL f^\stella \RR\HHom_{\A^n_{\FF_1}}^\pun(F,G)\to \RR\HHom_S^\pun( \LL f^\stella F,\LL f^\stella G).\]
}
\end{cosa}

\begin{thm} Let $f\colon S\to\A^n_{\FF_1}$ be an algebraic morphism. For any $\M$ in $D(S)$ one has:
\[ \RR\Gamma(S,\M)=\RR\Gamma (\A^n_{\FF_1},\RR f_\stella\M).\]
 Moreover {\rm (Finiteness theorem)}:  If $S$ is proper over $k$, then $\RR f_*$ maps $D_c^b(S)$ into $D_c^b(\A^n_{\FF_1})$ .
\end{thm}
\begin{proof} The first result follows immediately from adjunction, because $\LL f^* k=f^* k =\OO_S$. For the finiteness result, if $\M\in D^b_c(S)$, then, for any $p\in\A^n_{\FF_1}$ one has
\[ (\RR f_\stella \M)_p= \RR\Hom_{\A^n_{\FF_1}}^\pun(k_{U_p}, \RR f_\stella \M) = \RR\Hom_{S}^\pun(\Lc_p, \M ) = \RR\Gamma(S,\M\otimes \Lc_p^{-1}).\] Since $S$ is proper over $k$, $\RR\Gamma(S,\M\otimes \Lc_p^{-1})$ belongs to $D^b_c(k)$ and we are done.
\end{proof}

\subsection{Flatness} 

\begin{defn} {\rm  An algebraic $k$-morphism $f \colon S\to \A^n_{\FF_1}$ is {\em flat } if $f^\stella\colon \Shv_k(\A^n_{\FF_1})\to\Qcoh(S)$ is exact.  If, in addition, for any finitely generated $F\in\Shv_{k-{\rm fg}}(\A^n_{\FF_1})$, $f^\stella F=0$ implies $F=0$, then we say that $f$ is {\em faithfully flat}.}
\end{defn}

\begin{rem}{\rm If $f \colon S\to \A^n_{\FF_1}$ is flat, then the morphisms $s_i\colon\Lc_i\to\OO_S$ are injective, because $s_i$ is obtained by applying $f^\stella$ to the inclusion $k_{U_i}\hookrightarrow k$. Hence a flat algebraic $k$-morphism is defined by a collection $D$ of effective Cartier divisors. In the next theorem we see what conditions a collection $D$ of divisors must satisfy in order to $f_D$ to be flat.}
\end{rem}

\begin{defn}\label{generalposition} {\rm Let us  recall the usual Koszul complex associated to a sequence of effective Cartier divisors. Let $E_1,\dots,E_r$ be effective Cartier divisors on $S$ and $\OO_S(-E_i)$ the ideal defining $E_i$. Let $\M=\oplus_{i=1}^r\OO_S(-E_i)$ and $\omega\colon\M\to\OO_S$ the natural morphism. Then the Koszul complex $\Kos_\pun(E_1,\dots,E_r)$ is the complex
\[ 0\to {\bigwedge}^r_{\OO_S}\M\overset{i_\omega}\longrightarrow \cdots \overset{i_\omega}\to {\bigwedge}^2_{\OO_S}\M \overset{i_\omega}\longrightarrow \M \overset{ \omega}\longrightarrow \OO_S\to 0\] where $i_\omega$ is the inner contraction with $\omega$. 

 We say that $D_1,\dots ,D_n$ are in {\em  general position} if for any subset $p=\{i_1,\dots,i_r\}$ of $\{1,\dots,n\}$ the Koszul complex   $\Kos_\pun(D_{i_1},\dots,D_{i_r})$ is a resolution of $\OO_{\underset{i\in p}\bigcap D_i}$; in other words, for each $t\in \underset{i\in p}\bigcap D_i$, if   $f_{i_j}\in\OO_{S,t}$ is a generator of the ideal $\OO_S(-D_{i_j})$ at $t$, then $f_{i_1},\dots, f_{i_r}$ is a regular sequence in $\OO_{S,t}$.

If  $S$ is a $k$-scheme, we say that $D_1, \dots ,D_n$ {\em have $k$-flat intersections} if  $\underset{i\in p}\bigcap D_i$ is flat over $k$ for any $p$.}
\end{defn}

\begin{prop}\label{generic-open} Let $f \colon S\to\A^n_{\FF_1}$ be the continuous map defined by $n$ effective Cartier divisors  $D_1,\dots ,D_n$ on $S$. If these divisors   are in general position, then:\begin{enumerate} \item  $f \colon S\to \A^n_{\FF_1}$ is   open. \item A point $p\in\A^n_{\FF_1}$ belongs to the image of $f $ if and only if $\underset{i\notin p}\bigcap D_i$ is not empty.
\item $f $ is surjective if and only if $D_1\cap\cdots\cap D_n\neq \emptyset$.
\end{enumerate}
\end{prop} 

\begin{proof} (1) Let us first prove that $\Ima f$ is open. For each $p\in\A^n_{\FF_1}$, $f^{-1}(C_p)=\bigcap_{i\notin p} D_i$. If $f^{-1}(C_p)$ is not empty, then, for any $p'<p$,  $f^{-1}(C_{p'})$   is nowhere dense in $f^{-1}(C_p)$, because $D_1,\dots,D_n$ are in general position. Let us conclude that $\Ima f$ is open. Let $p\in\Ima f$ and $q>p$. We have to prove that $q\in\Ima f$. One has that $f^{-1}(C_q)$ is not empty, because it contains $f^{-1}(p)$. Since $\bigcup_{q'<q}f^{-1}(C_{q'})$ is nowhere dense in $f^{-1}(C_q)$, one has that
$$f^{-1}(q)=f^{-1}(C_q)-(\bigcup_{q'<q}f^{-1}(C_{q'})) $$ is not empty, i.e., $q\in\Ima f$.

Now, let $U$ be an open subset of $S$, $i\colon U\hookrightarrow S$ the inclusion. Then $f\circ i\colon U\to\A^n_{\FF_1}$ is the continuous map defined by the collection of Cartier divisors $D_{\vert U}$. Since $D_{\vert U}$ are in general position, $\Ima (f \circ i) $ is an open subset, i.e., $f(U)$ is open.

(2) If $p$ belongs to $\Ima f$, then obviously $f^{-1}(C_p)$ is not empty; conversely, if $f^{-1}(C_p)$ is not empty, then $C_p\cap\Ima f$ is not empty; since this is an open subset of $C_p$, it must contain its generic point $p$. 

(3) Since $\Ima f$ is open, $f$ is surjective if and only if $\0\in\Ima f$, and one concludes because $f^{-1}(\0)=D_1\cap\dots\cap D_n$.
\end{proof}

\begin{thm}\label{thetheorem} Let $S$ be a   $k$-scheme and $f \colon S\to\A^n_{\FF_1}$ the algebraic $k$-morphism defined by a collection  $D=\{ D_1,\dots,D_n\}$ of effective Cartier divisors on $S$. Let  us denote $V=\Ima f$.
The following conditions are equivalent:
\begin{enumerate} \item $f $ is flat.
\item $S$ is flat over $k$ and $D_1,\dots,D_n$ are in general position and have $k$-flat intersections.
\end{enumerate} 
Moreover, if any of these equivalent conditions hold, then:

\begin{enumerate}\item[(a)] If $F\in\Shv_k(\A^n_{\FF_1})$ is $k$-flat, then $f^\stella F$ is also    $k$-flat.
\item[(b)] The following conditions are equivalent:
\begin{enumerate}\item[(b.1)] $f $ is faithfully flat  on $V$, i.e.,  for any finitely generated sheaf $F\in\Shv_{k-{\rm fg}}(\A^n_{\FF_1})$ one has: $f^\stella (F)=0\Leftrightarrow F_{\vert V}=0$.
\item[(b.2)] For any $p\in V$, the scheme $\bigcap_{i\notin p}  D_i$ is   faithfully flat over $k$.
\end{enumerate}
\end{enumerate}
%
%
\end{thm}

\begin{proof}  First notice that,  from the very definitions, one has:
\begin{equation}\label{Koszul} f^\stella(\Kos_\pun(k_{C_p}))=\Kos_\pun(D_{i_1},\dots, D_{i_r}),\end{equation}  with $\{i_1,\dots,i_r\}=\{i:i\notin p\}$.

Let us first see that $(1)\Rightarrow (a)$. Assume that $f_D$ is flat  and let $F$ be a $k$-flat sheaf on $\A^n_{\FF_1}$.
 The composition  $$\aligned \C_{k-\text{mod}}&\longrightarrow \Shv_k(\A^n_{\FF_1})\overset{f^\stella}\longrightarrow \Qcoh(S)\\ E&\longmapsto E\otimes_kF\endaligned$$ is exact and it coincides with  $E\mapsto E\otimes_k f^\stella F$, because $f^\stella$ is $k$-linear; hence $f^\stella F$ is flat over $k$. 
 
Let us prove $(1)\Rightarrow (2)$. Since $f $ is flat  and $\OO_S=f^\stella(k)$, we conclude that $S$ is flat over $k$ by $(a)$.
Moreover, for any $p\in\A^n_{\FF_1}$, $\Kos_\pun(k_{C_p})$ is a resolution of $k_{C_p}$ and then $f^\stella (\Kos_\pun(k_K))$ is a resolution of $\OO_{S_{C_p}}$. We conclude that $D_1,\dots,D_n$ are in general position by equation \eqref{Koszul}. Finally, $D_1,\dots,D_n$ have $k$-flat intersections because $\OO_{\underset{i\notin p}\bigcap D_i}=f^\stella (k_{C_p})$.

Now, let us see $(2)\Rightarrow (1)$.   To prove that $f^\stella$ is exact, we have to prove that any sheaf $F$ on $\A^n_{\FF_1}$ is $f^\stella$-acyclic. We proceed in several steps: 

$(1.0)$ For any $p\in\A^n_{\FF_1}$, and any $k$-module $E$, $E\otimes_k k_{U_p}$ is $f^\stella$-acyclic, because $\LL f^\stella(E\otimes_k k_{U_p})=E\overset\LL\otimes_k f^\stella (k_{U_p}) = E \otimes_k f^\stella (k_{U_p})$, because $f^\stella (k_{U_p})$ is an invertible $\OO_S$-module and then, since $S$ is flat over $k$, it is $k$-flat. 


$(1.1)$ $E\otimes_k k_{C_p}$ is $f^\stella$-acyclic,  for any $k$-module $E$. Indeed, $E\otimes_k  \Kos_\pun(k_{C_p})$ is a resolution of $E\otimes_kk_{C_p}$ by $f^\stella$-acyclic sheaves, so $$\LL f^\stella (E\otimes_k k_{C_p})\simeq  f^\stella(E\otimes_k \Kos_\pun(k_{C_p})) =  E\otimes_k f^\stella( \Kos_\pun(k_{C_p})).$$  By equation \eqref{Koszul}, $f^\stella(\Kos_\pun(k_{C_p}))$ is a resolution of $\OO_{S_K}$, because $D_1,\dots,D_n$ are in general position. Moreover, it is a resolution by locally free $\OO_S$-modules (which are flat over $k$) and $\OO_{S_K}$ is $k$-flat because $D_1,\dots,D_n$ have $k$-flat intersections. Hence $E\otimes_k f^\stella (\Kos_\pun(k_{C_p}))$ is a resolution of $E\otimes_k\OO_{S_K}$ and we are done.


$(1.2)$ For any sheaf $F$, $C^\pun F$ is a finite right resolution of $F$ by $f^\stella$-acyclic sheaves, because each $C^iF$ is a finite product  of sheaves of the form $E\otimes_k k_{C_p}$. Hence $F$ is $f^\stella$-acyclic.


Now, let us prove  ${\rm (b.1)\Leftrightarrow (b.2)}$ when $f $ is flat. For any sheaf $F$ on $\A^n_{\FF_1}$ one has (see \ref{abierto-cerrado}, (b))
\begin{equation}\label{support} \supp(f^\stella F)\subseteq f^{-1}(\overline{\supp(F)}).\end{equation}

${\rm (b.2)\Rightarrow (b.1)}$. If   $F_{\vert V}=0$, then $\supp(F)\subseteq \A^n_{\FF_1}-V$ and then $\overline{\supp(F)})\subseteq \A^n-V$, because $V$ is open. By \eqref{support}, $\supp( f^\stella F)=\emptyset$, i.e., $f^\stella F=0$.

Now let $F$ be a finitely generated sheaf of $k$-modules such that $F_{\vert V}\neq 0$ and let us prove that $f^\stella F\neq 0$. Since $f^\stella (F_{V})=f^\stella F$ (because $f^\stella $ is exact and $f^\stella (F_{\A^n_{\FF_1}-V})=0$), we may assume that $\supp(F)\subseteq V$.

$(*)$ Let us first assume that $k$ is a field. Let $q_1,\dots, q_l$ be the generic points of $\supp(F)$ and let us consider the natural morphism $F\to M=\bigoplus_{q_i} F_{q_i}\otimes_k k_{C_{q_i}}$, whose cokernel is denoted by $N$. Notice that $N_{q_i}=0$. Hence, if $p_1,\dots , p_s$ are the generic points of $\supp(N)$, each $p_i$ is strictly smaller than some $q_j$. Hence $$\supp(f^\stella N)\subseteq f^{-1}( \overline{\supp(N)})\subsetneq \bigcup_{q_i} C_{q_i}= \supp(f^\stella M)$$ and then $f^\stella F\neq 0$, because $f^\stella M\to f^\stella N$ is not an isomorphism.

Now let $k$ be a ring and let $g\colon S\to X=\Spec k$ be the structure morphism. We are going to see that the result holds because it does fibrewise over $X$. For each $x\in X$, let $S_x=g^{-1}(x)$ be the schematic fibre and $i_x\colon S_x\hookrightarrow S$ the natural morphism.
 Let us denote $D_i^x:=D_i\cap S_x$, which are effective divisors on $S_x$, because $D_i\to X$ is  flat.
 
 $S_x$ is a ascheme over $k(x)= $ residue field of $x$. Let      $f_{D^x}\colon S_x\to\A^n_{\FF_1}$ be the algebraic $k(x)$-morphism  associated to the divisors $D_i^x$ on $S_x$;  the associated continuous map is $f_x=i_x\circ f$. Moreover, the intersection $\bigcap_{i\notin p} D_i^x$ is non empty if and only if $\bigcap_{i\notin p} D_i $ is so, because $\bigcap_{i\notin p} D_i\to X$ is faithfully flat. It follows that $\Ima f_x=\Ima f$ and that $D_1^x,\dots,D_n^x$ are in general position, so $f_{D^x}$ is flat. Now let $0\neq F$ be a finitely generated sheaf of $k$-modules supported on $V$. Then $F\otimes_kk(x)\neq 0$ for some $x\in X$,  and $F\otimes_kk(x)$ is supported on $V=\Ima f_x$, so $f_{x}^\stella (F\otimes_kk(x))\neq 0$ by $(*)$. Since   $f_{x}^\stella (F\otimes_kk(x))=i_x^*f^\stella F$, we conclude that $f^\stella F\neq 0$. 
 

Finally, let us prove that ${\rm (b.1)\Rightarrow (b.2)}$. Let $p\in V$. Then $S_{C_p}=\underset{i\notin p}\bigcap D_i$ is not empty by Proposition \ref{generic-open} and $S_{C_p}\to X=\Spec k$ is flat by (2). It only remains to prove that $S_{C_p}\to X$ is surjective. Let $x\in X$. By the hypothesis ${\rm (b.1)}$, $0\neq f^\stella (k(x)\otimes _k k_{C_p})= k(x)\otimes_k \OO_{S_{C_p}}$, so the fibre of $x$ by $S_{C_p}\to X$ is not empty.
\end{proof}

\begin{cor}\label{thetheorem2} Let $S$ be a    scheme over a field $k$ and $f \colon S\to\A^n_{\FF_1}$ the algebraic $k$-morphism defined by a collection  $D=\{ D_1,\dots,D_n\}$ of effective Cartier divisors on $S$. Then
\begin{enumerate} \item $f $ is flat   if and only if $D_1,\dots, D_n$ are in general position.
\item If $f $ is flat, then it is faithfully flat  on $V=\Ima f $, i.e., for any finitely generated sheaf $F\in\Shv_{k-{\rm fg}}(\A^n_{\FF_1})$ one has: $f^\stella (F)=0\Leftrightarrow F_{\vert V}=0$.
\end{enumerate}
\end{cor}

\end{exs}

\section{Cohen-Macaulayness on $\A^n_{\FF_1}$}\label{CM-section} In this section we make a brief sheaf-theoretic study of Cohen-Macaulayness on $\A^n_{\FF_1}$, via dualizing complexes, following Grothendieck's ideas (\cite{Hartshorne}). Thus, we shall see that there is a canonical sheaf $\omega_{\A^n_{\FF_1}}$ which is a dualizing sheaf for $D_\text{perf}(\A^n_{\FF_1})$ (i.e., induces reflexivity) and dualizes local cohomology. Then, for any closed subset $K\subseteq \A^n_{\FF_1}$, we shall obtain a canonical complex $\omega_K^\pun$ which induces reflexivity and dualizes local cohomology. Finally we shall prove that classic Cohen--Macaulayness of $K$ (defined in terms of the homology of links) agrees with Grothendieck's. For a more systematic treatment on arbitrary finite posets, see \cite{ST}.

Let us first recall the standard definitions of relative Cohen-Macaulayness on schemes.

\begin{defn}{\rm Let $S$ be a scheme over $k$. $S$ is said to be {\em relatively Cohen-Macaulay} (resp. {\em relatively Gorenstein}) if $S\to\Spec k$ is flat  and has Cohen-Macaulay (resp. Gorenstein) fibres.
 A coherent $\OO_S$-module $\M$ is called {\em relatively Cohen-Macaulay} if it is flat over $k$ and fibrewise Cohen-Macaulay.}
\end{defn}

\begin{rem}\label{CM-schemes}{\rm Let us recall the following well known criterion of Cohen-Macaulayness: Assume  that $S$ is a relatively Gorenstein $k$-scheme and let $Y$ be a closed subscheme of $S$. Then $Y$ is relatively Cohen-Macaulay if and only if it is flat over $k$ and there exists an integer $d$ such that $\HExt^i_{\OO_S}(\OO_Y,\OO_S)=0$ for any $i\neq d$ and $\HExt^d_{\OO_S}(\OO_Y,\OO_S)$ is flat over $k$. More generally, a coherent $\OO_S$-module $\M$ is relatively Cohen-Macaulay   if and only if it is flat over $k$ and there exists an integer $d$ such that $\HExt^i_{\OO_S}(\M,\OO_S)=0$ for any $i\neq d$ and $\HExt^d_{\OO_S}(\M,\OO_S)$ is flat over $k$.}
\end{rem}

\subsubsection{Canonical sheaf of $\A^n_{\FF_1}$}$\,$\medskip

 Let us fix some notations. Let us denote $D(k)$ the derived category of complexes of $k$-modules, $D_c(k)$ the full subcategory of complexes with finitely generated cohomologies and $\Perf (k)$ the full subcategory of perfect complexes of $k$-modules. For any $E\in D(k)$, we shall denote  $$E^\vee:=\RR\Hom_k^\pun(E,k).$$ If $E$ is perfect, then $E^\vee$ is also perfect  and the natural morphism $E\to E^{\vee\vee}$ is an isomorphism. 

We shall denote by $D(\A^n_{\FF_1})$ the derived category of complexes of sheaves of $k$-modules on $\A^n_{\FF_1}$ and $D_c(\A^n_{\FF_1})$ the full subcategory of complexes with finitely generated cohomology sheaves. 
\begin{defn} {\rm A complex $F\in D (\A^n_{\FF_1})$ is called {\em pointwise perfect} if $F_p\in\Perf(k)$ for any $p\in \A^n_{\FF_1}$. Any finitely generated  and flat sheaf  is pointwise perfect (assuming $k$ is noetherian). We shall denote by $D_{\text{perf}}(\A^n_{\FF_1})$ the full subcategory of $D(\A^n_{\FF_1})$ whose objects are the pointwise perfect complexes.}
\end{defn}

\begin{defn}\label{canonical}{\rm The {\em canonical sheaf} of $\A^n_{\FF_1}$ (over $k$) is the sheaf of $k$-modules
\[ \omega_{\A^n_{\FF_1}}:=k_{\{\un\}}.\]}
\end{defn}

For any $F\in D (\A^n_{\FF_1})$ we shall denote $D(F):=\RR\HHom_{\A^n_{\FF_1}}(F,\omega_{\A^n_{\FF_1}})$. There is a natural morphism $F\to D(D(F))$.

\begin{lem}\label{duality-0} For any $p\in\A^n_{\FF_1}$ and any $E\in D(k)$, one has
\begin{enumerate} \item $D(k_{U_p} )\simeq k_{U_{\widehat p}}$, where $\widehat p$ is the complement of $p$.
\item $D(E\otimes_kk_{U_p} )\simeq E^\vee\otimes_k k_{U_{\widehat p}}$.
\end{enumerate}
\end{lem}

\begin{proof} (1) Since $p+\widehat p=\un$, we have an isomorphism $k_{U_{\widehat p}}\otimes_k k_{U_p}= k_{\{\un\}}$ and then a morphism  $k_{U_{\widehat p}}\to D(k_{U_p})$. Let us see that it is an   isomorphism stalkwise; for any $x\in \A^n_{\FF_1}$ one has
\[ \aligned D(k_{U_p})_x & = \RR\Hom_{U_x}^\pun((k_{U_p})_{\vert U_x}, k_{\{\un\}})= \RR\Hom_{U_x}^\pun(k_{U_{p+x}} , k_{\{\un\}})= (k_{\{\un\}})_{p+x}\\ &=\left\{\aligned k, &\text{ if } p+x=\un\\ 0, &\text{ otherwise}\endaligned \right. = \left\{\aligned k, &\text{ if } x\geq \widehat p\\ 0, &\text{ otherwise}\endaligned \right. =(k_{U_{\widehat p}})_x.    \endaligned\]

(2) One has a natural morphism $ E^\vee\otimes_k D(k_{U_p})\to D(E\otimes_kk_{U_p} ) $. Let us see that it is an isomorphism stalkwise. Since $D(E\otimes_k k_{U_p})\simeq \RR\HHom_{\A^n_{\FF_1}}^\pun(E,D(k_{U_p}))$, one has
\[\aligned  D(E\otimes_kk_{U_p} )_x&  = \RR\Hom_{U_x}^\pun(E, (k_{U_{\widehat p}})_{\vert U_x}) = \RR\Hom_k^\pun(E, (k_{U_{\widehat p}})_x)\\ & = 
 \left\{\aligned E^\vee, &\text{ if } x\in U_{\widehat p}\\ 0, &\text{ otherwise}\endaligned \right. = (E^\vee\otimes_k k_{U_{\widehat p}})_x.\endaligned\]
\end{proof}

\begin{prop}\label{duality-1}   For any $F\in D_\text{\rm perf}(\A^n_{\FF_1})$ one has: \begin{enumerate} \item $D(F)\in D_\text{\rm perf}(\A^n_{\FF_1})$. \item The natural morphism
\[ F\to D(D(F))\] is an isomorphism.\end{enumerate}

If $k$ is Gorenstein  of finite Krull dimension, the same statements hold for  $ D_c(\A^n_{\FF_1})$. 
\end{prop}

\begin{proof} If $F=E\otimes_k k_{U_p}$ for some $E\in \Perf(k)$, $p\in \A^n_{\FF_1}$, the result follows from Lemma \ref{duality-0}. Now, for a general $F\in D_\text{perf}(\A^n_{\FF_1})$, the result holds for $C_iF$, $i=0,\dots ,n$, because $C_iF$ is a finite direct sum of complexes of type $E\otimes_kk_{U_p} $, $E\in \Perf(k)$. One concludes from the exact sequence of complexes $0\to C_nF\to\cdots\to C_0F\to F\to 0$. The same proof works for $F\in D_c(\A^n_{\FF_1})$ if $k$ is Gorenstenin of finite Krull dimension, because, in this case,  for any $E\in D_c(k)$ one has that $E^\vee\in D_c(k)$   and $E\to E^{\vee\vee}$ is an isomorphism. 
\end{proof}

\begin{prop} \label{duality-2} Let $F\in D(\A^n_{\FF_1}) $. One has:
\medskip

{\rm (1)} $\RR\Hom^\pun_{\A^n_{\FF_1}}(F, \omega_{\A^n_{\FF_1}})=   \RR\Gamma_\0(\A^n_{\FF_1},F)^\vee[-n]$. 

Hence, if $F$ is a sheaf and $k$ is a field: $$\Ext^i_{\A^n_{\FF_1}}(F,\omega_{\A^n_{\FF_1}})=H^{n-i}_\0(\A^n_{\FF_1},F)^*.$$ 

{\rm (2)} For any $p\in\A^n_{\FF_1}$ one has
\[ \left[ \RR\HHom^\pun_{\A^n_{\FF_1}} (F,\omega_{\A^n_{\FF_1}})\right]_p= \RR\Gamma_p(U_p,F)^\vee [-d_p] .\]
with $d_p=\dim U_p$. If $F$ is a sheaf and $k$ is a field,
\[ \left[\HExt^i_{\A^n_{\FF_1}}(F,\omega_{\A^n_{\FF_1}})\right]_p= H^{d_p-i}_p(U_p,F)^*.\]
\end{prop}

\begin{proof} On the one hand, from the exact sequence $0\to F_{\PP^{n-1}_{\FF_1}}\to F\to F_{\{\0\}}\to 0$ we obtain the exact triangle
\begin{equation} \label{eq} \RR\Hom^\pun_{\A^n_{\FF_1}}( F , \omega_{\A^n_{\FF_1}})\to  \RR\Hom^\pun_{\A^n_{\FF_1}}( F_{\PP^{n-1}_{\FF_1}} , \omega_{\A^n_{\FF_1}})\to \RR\Hom^\pun_{\A^n_{\FF_1}} ( F_{\{\0\}}, \omega_{\A^n_{\FF_1}})[1].\end{equation}
On the other hand, from the exact triangle of local cohomology $$\RR\Gamma_\0(\A^n_{\FF_1},F)\to \RR\Gamma (\A^n_{\FF_1},F)\to \RR\Gamma (\PP^{n-1}_{\FF_1},F)$$ one obtains the exact triangle
\begin{equation}\label{eq2} \RR\Gamma_\0(\A^n_{\FF_1},F)^\vee[-1]\to \RR\Gamma (\PP^{n-1}_{\FF_1},F)^\vee\to \RR\Gamma (\A^n_{\FF_1},F)^\vee.\end{equation} Comparing \eqref{eq} and \eqref{eq2}, we are reduced to prove:

(a) $\RR\Gamma (\PP^{n-1}_{\FF_1},F)^\vee[1-n]=\RR\Hom^\pun_{\A^n_{\FF_1}}( F_{\PP^{n-1}_{\FF_1}} , \omega_{\A^n_{\FF_1}})$.

(b) $\RR\Gamma (\A^n_{\FF_1},F)^\vee[-n]= \RR\Hom^\pun_{\A^n_{\FF_1}}( F_{\{\0\}}, \omega_{\A^n_{\FF_1}})$.

For (a), $$\RR\Hom_{\A^n_{\FF_1}}^\pun( F_{\PP^{n-1}_{\FF_1}} , \omega_{\A^n_{\FF_1}})= \RR\Hom_{\PP^{n-1}_{\FF_1}}^\pun( F_{\vert \PP^{n-1}_{\FF_1}} , k_{\{\un\}}) = \RR\Hom_k^\pun(\RR\Gamma(\PP^{n-1}_{\FF_1},F),k)[1-n],$$  where the last equality follows from duality, because $D_{\PP^{n-1}_{\FF_1}}=k_{\{\un\}}[n-1]$.

For (b), let us first see that \begin{equation}\label{eq3}\RR\HHom_{\A^n_{\FF_1}}^\pun(k_{\{\0\}}, \omega_{\A^n_{\FF_1}})=k_{\{\0\}}[-n].\end{equation}  It is immediate that $\RR\HHom_{\A^n_{\FF_1}}^\pun(k_{\{\0\}}, \omega_{\A^n_{\FF_1}})$ is supported at $\0$, so we are reduced to prove that the stalk at $\0$ is $k[-n]$. Since the stalk at $\0$ are the global sections, we have to prove that $\RR\Hom_{\A^n_{\FF_1}}^\pun(k_{\{\0\}}, \omega_{\A^n_{\FF_1}})=k[-n]$, i.e., $\RR\Gamma_\0(\A^n_{\FF_1},\omega_{\A^n_{\FF_1}})=k[-n]$. Form the exact sequence $0\to \omega_{\A^n_{\FF_1}}\to k\to k_{\A^n_{\FF_1}-\un}\to 0$ we obtain that $\RR\Gamma_\0(\A^n_{\FF_1},\omega_{\A^n_{\FF_1}})= \RR\Gamma_\0(\A^n_{\FF_1},k_{\A^n_{\FF_1}-\un})[-1]$, because $\RR\Gamma_\0(\A^n_{\FF_1},k )=0$ (since $\A^n_{\FF_1}$ and $\A^n_{\FF_1}-\0$ are contractible). Now, $\RR\Gamma_\0(\A^n_{\FF_1},k_{\A^n_{\FF_1}-\un})= \RR\Gamma_\0(\A^n_{\FF_1}-\un,k ) $ which is computed from the exact triangle of local cohomology (recall that $\HH^{n-2}=\A^n_{\FF_1}-\{\0,\un\}$)
\[ \RR\Gamma_\0(\A^n_{\FF_1}-\un,k )\to \underset {\underset {\displaystyle k}\parallel}{\RR\Gamma (\A^n_{\FF_1}-\un,k )} \to  \underset {\underset {\displaystyle k\oplus k[2-n]}\parallel} {\RR\Gamma (\HH^{n-2},k )} \] which yields that $ \RR\Gamma_\0(\A^n_{\FF_1}-\un,k )=k[1-n]$ and then \eqref{eq3} is proved.

Now, let us prove (b). Since $F_{\{\0\}}=F\otimes_kk_{\{\0\}}$, one has $$\aligned \RR\Hom^\pun_{\A^n_{\FF_1}}( F_{\{\0\}}, \omega_{\A^n_{\FF_1}})&= \RR\Hom^\pun_{\A^n_{\FF_1}}( F , \RR\HHom_{\A^n_{\FF_1}}^\pun(k_{\{\0\}}, \omega_{\A^n_{\FF_1}}) )\overset{\eqref{eq3}}= \RR\Hom^\pun_{\A^n_{\FF_1}}( F ,k_{\{\0\}})[-n]\\ &= \RR\Hom^\pun_k(F_\0,k)[-n] = \RR\Hom^\pun_k(\RR\Gamma(\A^n_{\FF_1},F),k)[-n]. \endaligned$$

\end{proof}

Using general results of \cite{ST} we obtain the following:

\begin{cor} {\rm (a)} For any open subset $V$ of $\A^n_{\FF_1}$, ${\omega_{\A^n_{\FF_1}}}_{\vert V}=k_{\{\un\}}$ is a canonical sheaf on $V$ (i.e., it satisfies Proposition \ref{duality-1}).

{\rm (b)} For any closed subset $K\overset i\hookrightarrow\A^n_{\FF_1}$, the complex
\[ \omega_K^\pun:=i^{-1}\RR\HHom_{\A^n_{\FF_1}}^\pun(k_K, \omega_{\A^n_{\FF_1}})=\HHom_{\A^n_{\FF_1}}(\Kos_\pun(k_K), \omega_{\A^n_{\FF_1}})_{\vert K}\] satisfies:

{\rm (1)} It is a canonical complex on $K$; that is, for any $F\in D_\text{\rm perf}(K)$, the natural morphism
\[ F\to \RR\HHom_K^\pun(\RR\HHom_K(F,\omega_K^\pun),\omega_K^\pun)\] is an isomorphism.

{\rm (2)} It dualizes local cohomology: for any $F\in D(K)$ one has an isomorphism
\[ \RR\Hom_K^\pun(F,\omega_K^\pun)=  \RR\Hom_k^\pun(\RR\Gamma_\0(K,F),k)[-n].\] 

 Also, the restriction of $\omega_K^\pun$ to $K^*:=K-\0$ is a canonical complex on $K^*$ and (shifted by $[-n]$) it is the global dualizing complex of $K^*$  (i.e., it dualizes cohomology).
\end{cor}

The following definition is inspired by its schematic analogue (see Remark \ref{CM-schemes}):

\begin{defn}\label{CM-sheaf} {\rm Let $F$ be a finitely generated sheaf on $\A^n_{\FF_1}$. We say that $F$ is {\em Cohen-Macaulay}  (over $k$) if $F$ is $k$-flat and there exists $d\in\NN$ such that $\HExt^i_{\A^n_{\FF_1}}(F,\omega_{\A^n_{\FF_1}})=0$ for any $i\neq d$ and $\HExt^d_{\A^n_{\FF_1}}(F,\omega_{\A^n_{\FF_1}}) $ is flat over $k$. By Proposition \ref{duality-2},  this is equivalent to say that there exists $d\in\NN$ such that: for any $p\in \A^n_{\FF_1}$, $H^j_p(U_p,F)=0$ for any $j\neq d_p-d$ and  $H^{d_p-d}_p(U_p,F) $ is a flat $k$-module.
    
More generally, if $V$ is an open subset of $\A^n_{\FF_1}$, then we say that $F$ is {\em Cohen-Macaulay  on $V$} (over $k$) if $F_{\vert V}$ is $k$-flat and there exists $d\in \NN$ such that $\HExt^i_{\A^n_{\FF_1}}(F,\omega_{\A^n_{\FF_1}})_{\vert V}=0$ for any $i\neq d$ and $\HExt^d_{\A^n_{\FF_1}}(F,\omega_{\A^n_{\FF_1}})_{\vert V} $ is flat over $k$.}\end{defn}

\subsubsection{Cohen-Macaulayness and links} 
Let $K\subseteq \A^n_{\FF_1}$ be a closed subset. We recall the standard notion of Cohen-Macaulayness of $K$ in terms of the homology of links and see that it coincides with the Cohen-Macaulayness of the sheaf $k_K$.

\begin{defn}\label{CM-link}{\rm For each $p\in K$, the {\em link of $p$} is:
\[ \Link_K(p)=\{ q\in K: q\neq \0, p+q\in K \text{ and } p\cdot q=\0\}\] which is a closed subset of $K^*=K\negmedspace -\negmedspace \0$.

We say that $K$ is {\em Cohen-Macaulay over $k$} if for any $p\in K$ one has:
\[ \widetilde H_i(\Link_K(p),k)=0 \text{ for any } i<\dim\Link_K(p).\]}
\end{defn}

\noindent{\bf Notation.} For each $p\in \A^n_{\FF_1}$, we shall denote $U_p^*:=U_p \negmedspace - \negmedspace \{p\}$.

\begin{lem}\label{link} One has an homeomorphism: $\Link_K(p)\simeq U_p^*\cap K $.
\end{lem}

\begin{proof} The maps 
\[  \aligned \Link_K(p)&\to U_p^*\cap K\\ q&\mapsto p+q\endaligned \qquad \aligned U_p^*\cap K&\to \Link_K(p) \\ q&\mapsto q-p\endaligned\] are continuous (since they are order preserving) and mutually inverse.
\end{proof}

\begin{lem} \label{Ext=homlink} Let $K\subseteq \A^n_{\FF_1} $ be a closed subset. Then
\begin{enumerate} \item
$\RR\Hom^\pun_{\A^n_{\FF_1}}(k_K, \omega_{\A^n_{\FF_1}})=\LL_{\text{\rm red}}(K^*, k)[1-n]$, where $K^*=K\negmedspace-\negmedspace \0$. Hence: $$\Ext^i_{\A^n_{\FF_1}}(k_K,\omega_{\A^n_{\FF_1}})=\widetilde H_{n-i-1}(K^*,k).$$ 
\item For each $p\in K$ one has
\[ (\omega_K^\pun)_p=\left[ \RR\HHom^\pun_{\A^n_{\FF_1}} (k_K,\omega_{\A^n_{\FF_1}})\right]_p=\LL_{\text{\rm red}}(U_p^*\cap K,k)[1-d_p].\]
with $d_p=\dim U_p$. Hence
\[ H^i((\omega_K^\pun)_p)= \left[\HExt^i_{\A^n_{\FF_1}}(k_K,\omega_{\A^n_{\FF_1}})\right]_p= \widetilde H_{d_p-i-1} (U_p^*\cap K,k)\overset{\ref{link}}= \widetilde H_{d_p-i-1} (\Link_K(p),k).\]
\end{enumerate}
\end{lem}

\begin{proof} (1) By Proposition \ref{duality-2}, we have to prove that $\RR\Gamma_\0(\A^n_{\FF_1},k_K)^\vee\simeq \LL_{\text{\rm red}}(K^*, k)[1]$. One has that $\RR\Gamma_\0(\A^n_{\FF_1},k_K)= \RR\Gamma_\0(K,k )$; taking dual $\,^\vee$ in the triangle of local cohomology
$$ \RR\Gamma_\0(K,k )\to \underset{\underset{\displaystyle k}\parallel}{\RR\Gamma (K,k )}\to \RR\Gamma (K^*,k )$$ we obtain the exact triangle
\[  \LL(K^*,k)\to k\to \RR\Gamma_\0(K,k )^\vee\] and the result follows.

(2) Again, by Proposition \ref{duality-2}, we have to prove that $\RR\Gamma_p(U_p,k_K)^\vee\simeq \LL_{\text{\rm red}}(U_p^*\cap K, k)[1]$. Since $U_p\simeq \A^{d_p}_{\FF_1}$, we conclude by (1).
\end{proof}

\begin{cor}\label{genericCM} If $q$ is a generic point of $K$, then $\left[\HExt^i_{\A^n_{\FF_1}}(k_K,\omega_{\A^n_{\FF_1}})\right]_q=\left\{ \aligned k, &\text{ for } i=d_q\\ 0, &\text{ otherwise}     \endaligned\right.$.
\end{cor}

\begin{proof} If $q$ is a generic point of $K$, then $U_q^*\cap K$ is empty.
\end{proof}

\begin{cor}\label{CM=CM} Let $d=\codim(K):=n-\dim K$. The following conditions are equivalent:
\begin{enumerate} \item $K$ is Cohen-Macaulay over $k$ (Definition \ref{CM-link}).
\item $\HExt^i_{\A^n_{\FF_1}}(k_K,\omega_{\A^n_{\FF_1}})=0$ for any $i\neq d$.
\item[(2')] $\omega_K^\pun=\omega_K [-d]$, for some sheaf $\omega_K$ on $K$.
\item $k_K$ is Cohen-Macaulay over $k$ (Definition \ref{CM-sheaf}).
\end{enumerate} 
\end{cor}

\begin{proof} For each $p\in K$, let us denote $m_p:=\dim (U_p^*\cap K)=\dim \Link_K(p)$. 

(1) $\Rightarrow $ (2). Assume that $K$ is Cohen-Macaulay over $k$. Then $K$ is pure (see \cite{BH}), so  $m_p=d_p-1-d$. Hence, for any $i\neq d$, and any $p\in K$,
\[ \left[\HExt^i_{\A^n_{\FF_1}}(k_K,\omega_{\A^n_{\FF_1}})\right]_p\overset{\ref{Ext=homlink}}= \widetilde H_{m_p+d-i} (U_p^*\cap K,k) =0.\] 

(2) $\Rightarrow $ (3). It is clear that $k_K$ is $k$-flat. It remains to see that $\HExt^d_{\A^n_{\FF_1}}(k_K,\omega_{\A^n_{\FF_1}})$ is flat over $k$; first notice that $K$ is pure by Corollary \ref{genericCM}, so $m_p=d_p-1-d$. Then, the stalk of $\HExt^d_{\A^n_{\FF_1}}(k_K,\omega_{\A^n_{\FF_1}})$ at  a point $p$  coincides with the highest homology group of the link of $p$, so the result follows from the following  fact: for any finite topological space $X$, if $ \widetilde H_{i}(X,k)=0$ for any $i<\dim(X)$, then 
$ \widetilde H_{\dim(X)}(X,k)=$ is a flat $k$-module - in fact, a finite projective $k$-module -, because $\LL_{\text{red}}(X,k)$ is a finite (right)-resolution of $ \widetilde H_{\dim(X)}(X,k)$ by flat $k$-modules (in fact, finite free modules).

(3) $\Rightarrow $ (1). By definition, there exists $d_0$ such that  $\HExt^i_{\A^n_{\FF_1}}(k_K,\omega_{\A^n_{\FF_1}})=0$ for any $i\neq d_0$; by Corollary \ref{genericCM}, $d_0=d_q$ for any generic point $q$ of $K$, so $K$ is pure, $d_0=d$ and $m_p=d_p-1-d$.  Then, for any $p\in K$ and any $i\neq m_p$, one has
\[  \widetilde H_{i} (U_p^*\cap K,k)\overset{\ref{Ext=homlink}}= \left[\HExt^{m_p+d-i}_{\A^n_{\FF_1}}(k_K,\omega_{\A^n_{\FF_1}})\right]_p=0.\] Finally, the equivalence of (2) and (2') is immediate.
\end{proof}

\begin{defn}\label{CM-open} {\rm Let $K\subseteq \A^n_{\FF_1} $ be a closed subset and $V\subseteq \A^n_{\FF_1}$ an open subset. We say that {\em $K$ is   Cohen-Macaulay (over $k$) on $V$} if for any $p\in K\cap V$ one has
\[ \widetilde H_i (\Link_K(p),k)=0,\text{ for any } i<\dim\Link_K(p).\]} \end{defn}

The same proof that Corollary \ref{CM=CM} yields:

\begin{cor}\label{CM=CM2} Let us denote $d=$ codimension of $K\cap V$ in $V$.  The following conditions are equivalent:
\begin{enumerate} \item $K$ is Cohen-Macaulay on $V$ (over $k$).
\item $\HExt^i_{\A^n_{\FF_1}}(k_K,\omega_{\A^n_{\FF_1}})_{\vert V}=0$ for any $i\neq d$.
\item $(\omega_K^\pun)_{\vert K\cap V}$ is a sheaf (shifted by $-d$).
\item $k_K$ is Cohen-Macaulay  on $V$ (over $k$).
\end{enumerate} 
\end{cor}

\section{Main results}\label{Main results}

Let $f\colon S\to \A^n_{\FF_1}$ be an algebraic $k$-morphism. For any $F,G$ in $D(\A^n_{\FF_1})$ we have a natural morphism 
\begin{equation}\label{fstellahom} \LL f^\stella\RR\HHom_{\A^n_{\FF_1}}^\pun(F,G) \to \RR\HHom_{S}^\pun(\LL f^\stella F, \LL f^\stella G) 
\end{equation}
which is not an isomorphism in general (it fails, for example, for $F=k_{U_p}, G=k_{U_q}$, with $p\ngeq q$). However,   it is an isomorphism for any pointwise perfect $F$ if we take $G=\omega_{\A^n_{\FF_1}}$:
 
\begin{thm}\label{extens} Let $S$ be a scheme over $k$ and $f \colon S\to \A^n_{\FF_1}$ an algebraic $k$-morphism. For any pointwise perfect complex $F\in D_\text{\rm perf}(\A^n_{\FF_1})$ one has:
\[ \LL f^\stella \RR\HHom^\pun_{\A^n_{\FF_1}}(F,\omega_{\A^n_{\FF_1}})= \RR\HHom^\pun_{S}(\LL f^\stella F, f^\stella\omega_{\A^n_{\FF_1}}).\] If $f$ is flat, then
\[ f^\stella \HExt^i_{\A^n_{\FF_1}}(F,\omega_{\A^n_{\FF_1}})= \HExt^i_{\OO_S}(f^\stella F, f^\stella \omega_{\A^n_{\FF_1}}).\]
\end{thm}

\begin{proof} Recall that $\omega_{\A^n_{\FF_1}}=k_{\{\un\}}$, so $\LL f^\stella (\omega_{\A^n_{\FF_1}})=  f^\stella (\omega_{\A^n_{\FF_1}})=\Lc_\un$, and for any $p\in\A^n_{\FF_1}$ one has \begin{equation}\label{p+q}\Lc_p\otimes_{\OO_S}\Lc_{\widehat p}=\Lc_\un.\end{equation}

(a) If $F=E\otimes_k k_{U_p}$, with $E\in\Perf(k)$, then 
\[\aligned   \LL f^\stella \RR\HHom^\pun_{\A^n_{\FF_1}}(F,\omega_{\A^n_{\FF_1}})& \overset{\ref{duality-0}}=  \LL f^\stella (E^\vee\otimes_k k_{U_{\widehat p}}) = E^\vee\otimes_k  f^\stella k_{U_{\widehat p}}= E^\vee\otimes_k \Lc_{\widehat p}\\ &  \overset{\eqref{p+q}}= E^\vee\otimes_k\RR\HHom^\pun_{S}(\Lc_p, \Lc_\un)\overset{E\text{ is perfect}}{===} \RR\HHom^\pun_{S}(E \otimes_k\Lc_p, \Lc_\un)\\ & = \RR\HHom^\pun_{S}( \LL f^\stella F, f^\stella\omega_{\A^n_{\FF_1}}).\endaligned\]

Now, for any pointwise perfect $F$, the result holds for $C_iF$ by (a). From the exact sequence of complexes 
\[ 0\to C_nF\to\cdots\to C_0F\to F\to 0\] one concludes.
\end{proof}

The importance of this theorem must not be underestimated. As we shall see, it is a crucial result for the preservation of Cohen--Macaulayness by $f^\stella$ (i.e., for Reisner's Criterion) and it is at the root of  the structure of the canonical complexes of Stanley--Resiner rings.

\begin{rem} {\rm Theorem \ref{extens} still holds if we replace $f$ by $f_U$ (resp. $f_K$) for an open subset $U$ of $\A^n_{\FF_1}$ (resp. a closed subset), taking into account \ref{abierto-cerrado}. In case of an open subset $U$,  $\omega_{\A^n_{\FF_1}}$ must be replaced by $(\omega_{\A^n_{\FF_1}})_{\vert U}$ and in case of a closed subset $K$ by $\omega_K^\pun$.}
\end{rem}

\subsubsection{Reisner's criterion}
 
\begin{thm}\label{ReisnerOnSchemes} Let $S$ be a scheme over $k$, $D=\{ D_1,\dots ,D_n\}$ effective Cartier divisors on $S$ and $f \colon S\to\A^n_{\FF_1}$ the  algebraic $k$-morphism defined by them. Let us denote $V=\Ima f$. Let us assume:

{\rm (a)} $S$ is flat over $k$.

{\rm (b)} $D_1,\dots,D_n$ are in general position and have $k$-flat intersections (Definition \ref{generalposition}).

{\rm (c)} For any $p\in V$, the intersection $\underset{i\notin p}\bigcap D_i$ is faithfully flat over $k$.

Now, let $K\subseteq\A^n_{\FF_1}$ be a closed subset and $S_K=f^{-1}(K)$ the associated closed subscheme of $S$. Let us denote by $d$ the codimension of $K\cap V$ in $V$. The following statements are equivalent:

\begin{enumerate}\item $K$ is Cohen-Macaulay over $k$ on $V$, that is: $$\widetilde H_i(\Link_K(p),k)=0,\text{ for any } p\in K\cap V \text{ and any }  i<\dim\Link_K(p).$$ 
\item $\HExt_{\OO_S}^i(\OO_{S_K},\OO_S)=0$ for any $i\neq d$. 
\end{enumerate}

Moreover, if $S$ is relatively Gorenstein, then these conditions are equivalent to:
\begin{enumerate}
\item[(3)] $S_K$ is relatively Cohen-Macaulay.
\end{enumerate}

More generally, for any finitely generated sheaf $F\in\Shv_{k-\text{\rm fg}}(\A^n_{\FF_1})$,  the following conditions are equivalent:

\begin{enumerate} \item[(1')] $F$ is Cohen-Macaulay  on $V$ (Definition \ref{CM-sheaf}).
\item[(2')] There exists a non negative integer $d_0$, such that $\HExt_{\OO_S}^i(f^\stella F,\OO_S)=0$ for any $i\neq d_0$ and $\HExt_{\OO_S}^{d_0}(f^\stella F,\OO_S)$ is flat over $k$.  
\item[(3')] $f^\stella F$ is relatively Cohen-Macaulay  (assuming that $S$ is relatively Gorenstein).
\end{enumerate}
\end{thm}

\begin{proof} The hypothesis (a)--(c) tell us (Theorem \ref{thetheorem}) that $f$ is flat   and faithfully flat  on $V$. Then, for any finitely generated sheaf $F$ on $\A^n_{\FF_1}$ one has:
\[\aligned  f^\stella F=0&\Leftrightarrow F_{\vert V}=0\\ f^\stella F\text{ is flat over } k &\Leftrightarrow F_{\vert V} \text{ is flat}.\endaligned\] Now, 
 the equivalence of (1') and (2') (and (3') if $S$ is relatively Gorenstein) follows from Theorem \ref{extens}, taking into account that $f^\stella \omega_{\A^n_{\FF_1}}$ is an invertible $\OO_S$-module. This same theorem, together with Corollary \ref{CM=CM2}, give the equivalence of (1) and (2)   (and (3) when $S$ is relatively Gorenstein).
\end{proof}

\begin{rem} {\rm Since Cohen-Macaulayness is a local question, the proof of  Theorem \ref{ReisnerOnSchemes} could have been reduced to the case that the divisors $D_i$ are principal (i.e., $\OO_S(-D_i)$ are trivial) and then, factorizing the morphism $S\to \A^n_{\FF_1}$ through a flat morphism of schemes $S\to \A^n_{k}$ and the ``classic'' morphism $\A^n_k\to \A^n_{\FF_1}$, to reduce the proof to the case $\A^n_k\to\A^n_{\FF_1}$. But in this case, the proof is not simpler that in the general case, so we have directly proved it in the general case.}
\end{rem}

\begin{rem} \label{remark-opensimplicial}{\rm In the case of the classic Stanley--Reisner correspondence we also have the following result. Let $K\subseteq\A^n_{\FF_1}$ be a simplicial complex and $K'$ a subcomplex (i.e., a closed subset). The complement $K-K'$ is an open subset, but it is no longer a simplicial complex. However, it still makes sense our definition of Cohen--Macaulayness: $K-K'$ is Cohen--Macaulay if $K$ is Cohen--Macaulay on $\A^n_{\FF_1}-K'$ (Definition \ref{CM-open}); this amounts to say that the canonical complex $\omega_{K-K'}^\pun:=(\omega_K^\pun)_{\vert K-K'}$ is a sheaf (placed in some degree). Then one has:
\[  K-K' \text{ is Cohen--Macaulay }\Leftrightarrow \text{ the scheme } S_K-S_{K'} \text{ is (relatively) Cohen--Macaulay}.\] Notice that there is a general notion of Cohen--Macaulayness on finite posets, due to Baclawski (\cite{Ba}). This notion, when applied to $K-K'$, does not agree with ours (it is a stronger condition), and hence it does not satisfy the above result. This says that our sheaf-theoretic definiton of Cohen-Macaulayness fits better with Stanley--Reisner theory than Baclawski's (see \cite{ST} for more about the comparison between these Cohen-Macaulayness notions).

 }

\end{rem}

%
%
%
%

\subsubsection{  Canonical complexes of Stanley--Reisner rings}\label{canonicalcomplexsubsection} For each simplical complex $K\subseteq\A^n_{\FF_1}$, we have an exact functor
\[\pi_K^\stella\colon\Shv_k(K)\longrightarrow\Qcoh(S_K).\] 
As a consequence of Theorem \ref{extens} we obtain the following:

\begin{thm}\label{canonicalcomplex} One has:
\[ \pi_K^\stella\,(\omega^\pun_K)=\omega^\pun_{S_K}:= \text{ \rm restriction to } S_K \text{ \rm of  }\RR\HHom_S^\pun(\OO_{S_K},\Omega^n_{\A^n_k/k}),\] which is a (relative) dualizing complex of $S_K$. In particular, if $K$ is Cohen--Macaulay, then
\[ \pi_K^\stella\,(\omega_K) = \omega_{S_K}:=\HExt_S^d(\OO_S,\Omega^n_{\A^n_k/k})\] which is the (relative) canonical sheaf of the (relative) Cohen--Macaulay scheme $S_K$.
\end{thm}

\begin{proof} Let us denote $i\colon K\hookrightarrow\A^n_{\FF_1}$, $i_S\colon S_K\hookrightarrow \A^n_k$ the natural inclusions. Then
\[ {i_S}_*\pi_K^\stella\,(\omega^\pun_K)=  \pi^\stella i_* (\omega_K^\pun) =   \pi^\stella\RR\HHom_{\A^n_{\FF_1}}^\pun (k_K,\omega_{\A^n_{\FF_1}}) \overset{\ref{extens}} =  \RR\HHom_{\A^n_k}^\pun (\OO_{S_K},\Omega_{\A^n_k})= {i_S}_* 
\omega^\pun_{S_K}.\]
\end{proof}

Notice that $\omega_K^\pun$ is explicitely described by $\omega_K^\pun=\HHom_{\A^n_{\FF_1}}(\Kos_\pun(k_K),\omega_{\A^n_{\FF_1}})_{\vert K}$. Hence, Theorem \ref{canonicalcomplex} gives a precise combinatorial description of the canonical complex of a Stanley--Reisner ring, which clarifies that of \cite{Gr}. Of course, Reisner's criterion is also an immediate consequence of Theorem \ref{canonicalcomplex} (taking into account Corollary \ref{CM=CM}).

\subsubsection{Projective Stanley--Reisner correspondence}. Let us now consider the exact functor
\[ \pi^\stella\colon \Shv(\PP^{n}_{\FF_1})\to \Qcoh(\PP^n_k)\] induced by the standard coordinate hyperplanes of $\PP^n_k$. For each simplicial complex $K\subseteq \A^{n+1}_{\FF_1}$, let us denote $K^*:=K\negmedspace-\negmedspace\0$, which is a closed subset of $\PP^n_{\FF_1}$. Then $\pi^\stella(k_{K^*})=\OO_{\PP(S_K)}$, where $\PP(S_K)$ is the projectivization of the Stanley--Reisner scheme $S_K$. For each $p\in\PP^n_{\FF_1}$ one has
\[\pi^\stella(k_{U_p})\simeq \OO_{\PP^n_\ZZ}(-c_p-1),\quad c_p=\dim C_p.\] One has also an exact functor
\[ \pi_{K^*}^\stella\colon \Shv(K^*)\longrightarrow \Qcoh(\PP(S_K)).\]

Let us denote $\omega_{\PP^n_{\FF_1}}=k_{\{\un\}}$, which is the dualizing sheaf of $\PP^n_{\FF_1}$, i.e., $\omega_{\PP^n_{\FF_1}}[n]$ is the global dualizing complex of $\PP^n_{\FF_1}$. One has $\pi^\stella \omega_{\PP^n_{\FF_1}}=\omega_{\PP^n_{k}}\simeq\OO_{\PP^n_k}(-n-1)$.  Moreover:

\begin{prop}\label{projectivesimplicial} One has $$\RR\pi_\stella \omega_{\PP^n_{k}} = \omega_{\PP^n_{\FF_1}}.$$  Consequently, for any $F\in D_{\text{\rm perf}}(\PP^n_{\FF_1})$, one has:
\[ H^i(\PP^n_{\FF_1},F)=H^i(\PP^n_k,\pi^\stella F).\] Then, for any simplicial complex $K$ and any $G\in D^b_c(K^*)$:
\[ H^i(K^*,G)=H^i(\PP(S_K),\pi_{K^*}^\stella G).\] In particular, $H^i(K^*,k) = H^i( \PP(S_K), \OO_{\PP(S_K)})$.
\end{prop}

\begin{proof} For each $p\in \PP^n_{\FF_1}$ one has:
\[\aligned  [\RR\pi_\stella \omega_{\PP^n_k}]_p&=\RR\Hom_{\PP^n_{\FF_1}}^\pun(k_{U_p}, \RR\pi_\stella \omega_{\PP^n_k}) = \RR\Hom_{\PP^n_{k}}^\pun(\pi^\stella k_{U_p}, \omega_{\PP^n_k}) 
\\ &= \RR\Hom_{\PP^n_{k}}^\pun(\OO_{\PP^n_k}(-1-c_p), \OO_{\PP^n_k}(-1-n )) =\left\{
\aligned k, &\text{ if } p=\un
\\ 0, &\text{ otherwise},\endaligned\right. 
\endaligned\] hence $\RR\pi_\stella \omega_{\PP^n_k}=k_{\{\un\}}$, as wanted. Now, 
\[ \aligned & \RR\Gamma(\PP^n_{\FF_1}, F)^\vee = \RR\Hom_{\PP^n_{\FF_1}}^\pun(F,\omega_{\PP^n_{\FF_1}}[n]) = \RR\Hom_{\PP^n_{\FF_1}}^\pun(F,\RR\pi_\stella \omega_{\PP^n_k}[n]) = \RR\Hom_{\PP^n_{k}}^\pun(\pi^\stella F,  \omega_{\PP^n_k}[n])\\ & = \RR\Gamma(\PP^n_{k}, \pi^\stella F)^\vee \endaligned\]
and then $\RR\Gamma(\PP^n_{\FF_1}, F)=\RR\Gamma(\PP^n_{k}, \pi^\stella F)$, because they are both perfect complexes.

%
%
The last formula follows by taking $F=i_*G$, with $i\colon K^*\hookrightarrow \PP^n_{\FF_1}$ the inclusion.
\end{proof}

%
%

This Proposition says that any problem in combinatorics that can be formulated in a cohomological way (i.e. as the vanishing of a cohomology group of a certain sheaf) can be translated to a cohomological problem on the associated Stanley--Reisner projective scheme. If $k$ is a field and $F$ is a sheaf of finite dimensional vector spaces, the equality $H^i(\PP^n_{\FF_1},F)=H^i(\PP^n_k,\pi^\stella F)$ is related to \cite[Thm. 3.3]{Ya01}.

Finally, regarding canonical complexes, one can repeat the same results of subsection \ref{canonicalcomplexsubsection} for the projective case: the complex on $K^*$
\[ \omega_{K^*}^\pun:= \HHom_{\PP^n_{\FF_1}}(\Kos_\pun(\ZZ_{K^*}),\omega_{\PP^n_{\FF_1}})_{\vert K^*}=(\omega_K^\pun)_{\vert K^*}\]  is a canonical complex on $K^*$, i.e., it induces reflexivity on  $D_\text{perf}(K^*)$; moreover, $\omega_{K^*}^\pun$ dualizes  cohomology: for any $F\in D (K^*)$ one has an isomorphism
\[ \RR\Hom_{K^*}^\pun (F,\omega_{K^*}^\pun)= \RR\Hom_k^\pun (\RR\Gamma(K,F),k)[-n]\]  and   $\pi_{K^*}^\stella (\omega_{K^*}^\pun)$ is the global dualizing complex of the projective Stanley--Reisner scheme $\PP(S_K)$ (shifted by $[-n]$).  If $K^*$ is Cohen--Macaulay, then $\omega_{K^*}^\pun$ is just a sheaf,  named {\em canonical sheaf of } $K^*$, and   $\pi_{K^*}^\stella \omega_{K^*}=\omega_{\PP(S_K)}$ is the canonical sheaf of the Cohen--Macaulay Stanley--Reisner projective scheme $\PP(S_K)$.

\end{document}